\documentclass[11pt]{article}
\usepackage[utf8]{inputenc}
\usepackage{graphicx} 
\usepackage{amsthm}   
\usepackage{mathtools}
\usepackage{svg}
\usepackage{amsmath}  
\usepackage{amsfonts} 
\usepackage{setspace} 
\usepackage{cancel}   
\usepackage{ amssymb }
\usepackage{enumitem} 
\usepackage{float}    
\usepackage{mathrsfs} 
\usepackage{hyperref} 
\usepackage{cleveref} 
\usepackage{stackrel} 
\usepackage{empheq} 
\usepackage{xcolor} 
\usepackage{caption} 
\usepackage{cite}
\hypersetup{
    colorlinks,
    linkcolor={blue!50!black},
    citecolor={blue!50!black},
    urlcolor={blue!80!black}
}

\newtheorem{thm}{Theorem}[section]

\newtheorem{lem}[thm]{Lemma}

\theoremstyle{definition}

\newtheorem{remark}[thm]{Remark}

\title{Weakly-singular formulation of the Fractional Laplacian operator}
\author{Oscar P. Bruno$^*$ and Sabhrant Sachan\footnote{Computing and Mathematical Sciences, Caltech, Pasadena, CA 91125, USA}}

\begin{document}
\date{}
\maketitle
\begin{abstract}
  This paper presents a new formulation of the fractional Laplacian
  operator $(-\Delta)^s$ in $n$-dimensional space ($n \ge 1$). The
  proposed formulation expresses $(-\Delta)^s$ as a composition of the
  classical Laplace differential operator and a weakly singular
  integral operator---which can be used to reduce e.g. the Dirichlet
  problem for the fractional Laplacian to a weakly singular integral
  equation involving both volumetric and boundary integral
  operators. This reformulation is well suited for efficient and
  accurate numerical implementation. Although a full description of
  the associated high-order algorithm is deferred to a subsequent
  contribution, several numerical examples are included in this paper
  to demonstrate the high accuracy and computational efficiency
  achieved by the proposed approach.
\end{abstract}

\section{Introduction\label{intro}}

We are concerned with the Dirichlet problem
\begin{equation}\label{frac_lapl}
  \left\{\begin{aligned}
    (-\Delta)^{s} u=f & \quad \text{in } \Omega\\
    u=0 & \quad \text{in } \Omega^{c},
  \end{aligned}\right. \qquad 0< s < 1,
\end{equation}
for the fractional Laplacian operator $(-\Delta)^{s}$
over a bounded domain $\Omega\subset\mathbb{R}^n$, where~\cite{bucur2016}
\begin{equation} \label{eq:frac_op}
        (-\Delta)^s u(x) = c_{n,s} \text{ P.V.} \int_{\mathbb{R}^n} \frac{u(x)-u(y)}{|x-y|^{n+2s}}dy, \quad c_{n,s}=\frac{2^{2 s} s \Gamma(s+n / 2)}{\pi^{n / 2} \Gamma(1-s)},
\end{equation}
and where P.V. denotes the Cauchy principal value. The fractional
Laplacian operator (FL) and associated nonlocal partial differential
equations play central roles in numerous areas of science and
engineering~\cite{sun2018,dipierro2021,giusti2020,giusti2020_2,burch2011,antil2022,antil2017,chen2004,laskin2000,pozrikidis2018fractional,metzler2000random,silling2000reformulation,bobaru2010peridynamic,duo2017comparative,epps1803turbulence,burresi2012weak},
including the fractional Schr\"odinger
equation~\cite{covi2020,bogdan2019,ghosh2020,covi2022,covi2021}, the
fractional heat equation~\cite{bonforte2017}, and fractional
reaction-diffusion systems~\cite{owolabi2021}---which are used to
describe phenomena in quantum mechanics and general relativity,
anomalous diffusion, and chemical or biological processes,
respectively. These diverse contexts highlight the importance of the
FL operator $(-\Delta)^s$, which, as the infinitesimal generator of
isotropic $2s$-stable L\'evy processes~\cite{bertion1996}, provides a
natural fractional analogue of the Laplace operator.  Despite its
widespread relevance, solving equations involving the FL operator,
such as~\eqref{frac_lapl}, remains a challenging task.

The expression~\eqref{eq:frac_op} for the Fractional Laplacian
involves integration over an unbounded domain with a non-integrable
kernel, and it relies on the cancellation of $u(x)-u(y)$ as $x \to y$
for convergence---which, in a computational context,  is
prone to cancellation errors. Additionally the corresponding solutions
of~\eqref{frac_lapl} exhibit singularities at the domain boundary.
Consequently, previous numerical
discretizations~\cite{minden2020,burkardt2021,xu2020,duo2018,dwivedi2024}---including
those relying on singularity-subtraction techniques to treat the
kernel singularities as well as those based on mesh-free approaches
that evaluate the FL operator analytically for suitably chosen basis
functions, among others---exhibit slow convergence, large condition
numbers, severe cancellation errors, and, ultimately, stagnation of
convergence.  In particular, for generic smooth source functions $f$
encountered in practice, existing numerical algorithms for the
problem~\eqref{frac_lapl}---including methods based on adaptive finite
elements (FEM), radial basis functions, isogeometric analysis, and
Monte Carlo-based approaches---have generally exhibited low
convergence rates, stagnated convergence and limited accuracies for
generic smooth source functions $f$ in dimensions
$n\geq
2$~\cite{lischkle2020,xu2020,burkardt2021,minden2020,acosta2017short},
with errors on the order of $10^{-2}$--$10^{-3}$.  

This paper introduces new reformulations of the $n$D fractional
Laplacian operator~\eqref{eq:frac_op} and the associated boundary
value problem~\eqref{frac_lapl}. The new approach, which departs
significantly from previous methods, is based on a new expression
(eq.~\eqref{eq:frac-from-clas} below) for the FL
operator~\eqref{eq:frac_op} as a composition of the classical
Laplacian differential operator with a weakly singular integral
operator, together with a new {\em smooth unknown} $\phi$ defined over
the bounded domain $\Omega$, which is related to the unknown $u$
in~\eqref{frac_lapl} via the relation~\eqref{eq:FL-sol-ws} below. This
transformation thus reduces the original problem to an
integro-differential equation for an auxiliary unknown that remains
smooth up to the boundary. By inverting the Laplace operator, further,
a Fredholm-type integral equation involving only smooth unknowns and
weakly singular kernels is obtained---which includes single- and
double-layer boundary potentials as well as weakly singular integral
operators over $\Omega$.  This reformulation provides several
important advantages: (i)~It only involves weakly singular kernels;
(ii)~It utilizes unknowns that are smooth up to and including the
boundary $\partial \Omega$; and, (iii)~It only involves integrations
over bounded sets, as it only employs integral operators acting on
functions defined over $\Omega$ and $\partial\Omega$.  In particular
the new framework enables the development of numerical algorithms
that, as illustrated by a few numerical examples in
Section~\ref{sec:num-ex}, achieve rapid convergence and high accuracy
at modest computational cost---reaching up to thirteen digits of
accuracy with only a few thousand unknowns in $2$D domains.  (The
details of the numerical algorithm used, which additionally
incorporates a modified version of the rectangular-polar
method~\cite{bruno2020chebyshev, bruno2024direct}, lie beyond the scope of this paper
and will be reported elsewhere.)

As suggested above, the FL has emerged as a unifying element for
describing nonlocal phenomena across physics, mechanics, and
probability. In astrophysics, it provides a natural extension of
Newtonian gravity and MOND, allowing thin-disk galaxies and galactic
dynamics to be modeled through fractional Poisson equations that
account for long-range interactions beyond classical Laplacian
frameworks \cite{giusti2020,giusti2020_2}. At smaller scales, the same
operator underlies anomalous diffusion on bounded domains, where
nonlocal boundary conditions link analytic formulations to physical
transport processes \cite{burch2011}. In image processing, it has been
employed for image denoising, where the resulting PDE suppresses noise
while preserving essential image
features~\cite{antil2022,antil2017}. Its ability to capture
dissipation and dispersion mechanisms inaccessible to conventional
PDEs has enabled accurate models of lossy acoustic and electromagnetic
media, where frequency power-law attenuation arises naturally from
fractional spatial operators \cite{chen2004}. Fractional quantum
mechanics further illustrates the operator’s reach: by replacing
Brownian-motion-based path integrals with L\'evy flights, a fractional
Schr\"odinger equation is obtained that yields new dispersion
relations and particle models with applications ranging from condensed
matter to Quantum Chromodynamics \cite{laskin2000}. At a probabilistic
level, the operator is the infinitesimal generator of L\'evy
processes, providing the stochastic foundation for anomalous diffusion
and linking the analytic, physical, and probabilistic perspectives
\cite{pozrikidis2018fractional,metzler2000random}. Parallel
developments in continuum mechanics---particularly
peridynamics---offer a complementary integral framework for
nonlocality, conceptually aligned with fractional models; recent work
shows that, under suitable scalings, peridynamic operators converge to
the fractional Laplacian, establishing a mathematical bridge between
mechanics and probability
\cite{silling2000reformulation,bobaru2010peridynamic,duo2017comparative}. In
turbulence, fractional operators encode anomalous energy transfer and
dissipation across scales \cite{epps1803turbulence}; in optics, they
describe weak localization of light in L\'evy glasses and related
superdiffusive media \cite{burresi2012weak}.  

Multiple equivalent formulations of the operator~\eqref{eq:frac_op}
exist, as emphasized in comparative
reviews~\cite{mateusz2017,lischkle2020} (see
also~\cite{bhattacharyya2021}), each offering distinct advantages
depending on context. Generalizing the Fourier multiplier property of
the Laplacian to fractional orders, the classical definition for
functions on $\mathbb{R}^n$ is given in terms of the Fourier
transform,
\begin{equation}\label{eq:fourier_def}
(-\Delta)^s u(x) = \mathcal{F}^{-1}\left[|\xi|^{2s} \mathcal{F}[u](\xi)\right](x),
\end{equation}
where $\mathcal{F}$ and $\mathcal{F}^{-1}$ denote the Fourier and
inverse Fourier transforms, respectively; as shown in
\cite{mateusz2017}, the Fourier-based definition is equivalent to the
integral (Riesz) formulation expressed via the hypersingular
operator~\eqref{eq:frac_op}. On bounded domains, an alternative
definition, not equivalent to the Riesz version~\cite{Valdinoci2014},
is based on the set of eigenfunctions and eigenvalues of the classical
Laplace operator~\cite[Sec.~2.5.1]{lischkle2020}. As indicated above,
however, the Riesz formulation provides a natural probabilistic
interpretation in terms of random walks with jumps and associated
L\'evy processes.  A certain ``directional'' representation, in turn,
relates the FL~\eqref{eq:frac_op} to an average of one-dimensional
Riemann–Liouville fractional derivatives over all angular directions,
highlighting its connection to the Riesz derivative (see
e.g.~\cite{cai2019, li2020}). The aforementioned
reference~\cite{mateusz2017} provides ten equivalent definitions of
the FL~\eqref{eq:frac_op}---including distributional, semigroup, and
probabilistic interpretations; see also~\cite{duo2019study}. All of
the formulations just discussed present challenges from a
computational perspective, however. For example, in the context of
problem~\eqref{frac_lapl}, numerical approximations of the Fourier
transform of $u$ converge slowly due to the limited regularity of $u$
at the boundary of $\Omega$. This difficulty is further amplified by
the multiplication with $|\xi|^{2s}$ in the inverse
transform. Consequently, Fourier-based approaches, at least in their
raw form, are generally not computationally attractive.  In contrast,
as illustrated briefly in Section~\ref{sec:num-ex}, the formulation
introduced in this paper enables the development of numerical methods
with significantly improved convergence rates.  This enhancement is,
in fact, one of the primary motivations for the contributions
presented in this study.

A key aspect of the constructions presented in this paper concerns the
boundary regularity of the Poisson potential $V[f]$ for the Laplace
equation (eq.~\eqref{single_layer_vol}), as well as the corresponding
boundary regularity of a certain potential $F_s[\phi]$
(eq.~\eqref{eq:F_s-op}) which, for $0<s<1$, is related to but
different from the Riesz
potential~\cite{lu2012overdetermined,rozenblum2016isoperimetric}, and
which, for $s=0$, coincides with the potential $V[f]$
(Remark~\ref{rem:Gs}). Together with the
reformulation~\eqref{eq:frac-from-clas} in
Theorem~\ref{thm:FL_into_Lap}, the boundary-regularity results
presented in Lemmas~\ref{lem:boundary_reg_Fs}
and~\ref{lem:boundary_reg_V} for the operators $V[f]$ and $F_s[\phi]$
underly the theoretical foundations of the proposed reformulation for
the boundary value problem~\eqref{eq:frac_op}. As noted in
Remark~\ref{rem:Vf_reg}, to the authors’ knowledge, no previous
results have addressed the boundary regularity of either
operator---which may somewhat surprizing in the case of the well-known
Poisson potential.


This paper is organized as follows. Section~\ref{prelim} introduces
the necessary preliminaries, including existing regularity results as
well as volumetric and boundary single- and double-layer potentials
related to the Laplace and Poisson equations.  The proposed
reformulations of the fractional Laplacian operator and the associated
Dirichlet problem~\eqref{frac_lapl} and associated regularity results
are presented in Section~\ref{sec:Weakly_sing_form}
(Theorems~\ref{thm:FL_into_Lap} through~\ref{thm:new_form_1D} and
Lemmas~\ref{lem:boundary_reg_Fs} and~\ref{lem:boundary_reg_V}).
Illustrative numerical results, finally, are presented in
Section~\ref{sec:num-ex}.

\section{Preliminaries\label{prelim}}


This section introduces necessary conventions and notations, and it
reviews a number of known results that underly the theory presented in
this contribution.  Except for Theorem~\ref{thm:FL_into_Lap} and
Lemmas~\ref{lem:FL_op_grad} and~\ref{lem:und_int}, throughout this
paper it is assumed that $\Omega \subset \mathbb{R}^n$ is a (possibly
multiply connected) open, bounded, and connected domain with smooth
(infinitely differentiable) boundary, as illustrated in
Figure~\ref{Omega}. In particular, the complement of its closure,
$\overline{\Omega}^c := \mathbb{R}^n \setminus \overline{\Omega}$,
consists of a finite number $n_h+1$ of connected components, denoted
by $\Omega'_\ell$, $0 \leq \ell \leq n_h$ ($n_h \geq 0$):
\begin{equation}\label{domain_Omega}
  \overline{\Omega}^c = \mathbb{R}^n \setminus \overline{\Omega} = \bigcup_{\ell=0}^{n_h} \Omega'_\ell;
\end{equation}
the boundaries of the various sets in this equation are denoted by
$\partial\Omega$ and $\partial\Omega'_\ell$. Here, for each
$\ell = 1, \dots, n_h$, the set $\Omega'_\ell$ is a bounded connected
component, while $\Omega'_0$ denotes the unique unbounded connected
component of $\overline{\Omega}^c$.
\begin{figure}[H]
    \centering
    \includegraphics[width=14cm]{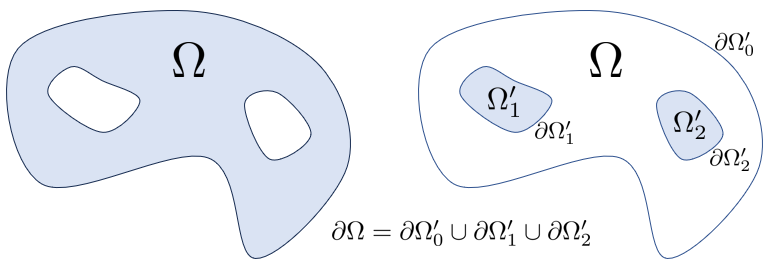}
    \caption{A multiply-connected domain with two holes\label{Omega}}
\end{figure}

Throughout this paper we write $\mathbb{N}_0 = \mathbb{N}\cup\{0\}$, and,
for $c\in\mathbb{N}$ and $\alpha\in \left(\mathbb{N}_0\right)^c$, we let
\begin{equation}\label{eq:mlt-indx-nrm}
  |\alpha| = \sum_{i=1}^{c}\alpha_i.
\end{equation}
For $c\leq n$ and $\alpha \in(\mathbb{N}_0)^c$, and for a function
$h = h(t_1,\dots,t_n)$ which is $|\alpha|$-differentiable with respect
to the first $c$ variables $(t_1,\dots, t_c)$, further, we denote by
\begin{equation}\label{eq:AP2_defn_Dalp}
  D^{\alpha}_t h = \dfrac{\partial^{|\alpha|}h}{\partial t_1^{\alpha_1}\cdots\partial t_c^{\alpha_{c}} }, \qquad \alpha \in \left(\mathbb{N}_0\right)^c.
\end{equation}
When necessary for clarity, variants of the
definition~\eqref{eq:AP2_defn_Dalp}, such as e.g.
$D_{\tau}^{\alpha}h(\tau,t)$ and $D_t^{\alpha}h(\tau,t)$, are also
used to specify derivatives of order
$\alpha = (\alpha_1,\dots,\alpha_c)$ of a function $h(\tau,t)$ with
respect to the variables $t$ and $\tau$, respectively.

For $t>0$, the following notations are
employed:
\begin{equation}
  \label{eq:limit-holder}
  C^{t+0} = \bigcup_{\varepsilon>0} C^{t+\varepsilon } \quad \text{and} \quad C^{t-0} = \bigcap_{0<\varepsilon \le t} C^{t-\varepsilon }.
\end{equation}
Here, for a real number $r\geq 0$, we have set
$C^r(\Omega) = C^{k+\delta}(\Omega)=C^{k,\delta}(\Omega)$ where $k$ denotes
the integer part of $r$ and where $\delta = r - k$. In other words,
$C^r(\Omega)$ denotes the space of $k$-times continuously
differentiable functions whose $k$-th order derivatives are H\"older
continuous of order $\delta$; see
e.g.~\cite{evans2022partial}. Furthermore, throughout this paper,
$d(x) \in C^{\infty}(\overline{\Omega})$ denotes a fixed function that
is strictly positive in $\Omega$ and such that the quotient
$d(x)/\mathrm{dist}(x,\partial\Omega)$ is infinitely differentiable in
a neighborhood of $\partial\Omega$ (up to and including the boundary
$\partial\Omega$), where $\mathrm{dist}(x,\partial\Omega)$ denotes the
distance from $x$ to $\partial\Omega$.
\begin{thm} \label{thm:fu_rel} Let $n\in\mathbb{N}$ and let
  $\Omega\in\mathbb{R}^n$ denote a bounded and open domain whose
  boundary is infinitely differentiable, and let $s\in (0,1)$. Then,
  for each function $f \in C^{t+0}(\overline{\Omega})$, $t \ge 0$,
  equation~\eqref{frac_lapl} admits a unique solution $u \in C(\mathbb{R}^n)$, which may
  additionally be expressed in the form
  \begin{equation}\label{sing_form}
    u(x) = d^s(x) \phi(x)\quad\mbox{for}\quad x \in \Omega,\quad (d^s(x) =(d(x))^s),
  \end{equation}
  for a certain function
  $\phi \in C^{t+s-0}(\overline{\Omega}) \cap
  C^{t+2s-0}(\Omega)$. Further, $f \in C^{\infty}(\overline{\Omega})$
  if and only if $\phi \in C^{\infty}(\overline{\Omega})$.
\end{thm}

\begin{proof}
  Since
  $f \in C^{t+0}(\overline{\Omega}) \subseteq C(\overline{\Omega})$,
  in view of~\cite[Thm. 4.6]{liu2025} it follows that there exists a
  unique solution $u \in C(\mathbb{R}^n)$ of
  equation~\eqref{frac_lapl}. Clearly,
  $u \in L^{\infty}(\Omega)$, and, thus, in view of the
  assumption $f \in
  C^{t+0}(\overline{\Omega})$,~\cite[Thm. 4]{grub2015} tells us that
  $u$ can be expressed in the form~\eqref{sing_form} with
  $\phi \in C^{t+s-0}(\overline{\Omega}) \cap
  C^{t+2s-0}(\Omega)$, as desired.
\end{proof}

The reformulation presented in Section~\ref{sec:Weakly_sing_form}
reduces the FL problem~\eqref{frac_lapl} to a sequence of two problems,
namely, a Poisson problem of the form
\begin{equation}\label{eq:poisson}
  \left\{\begin{aligned}
    \Delta v = g_1 & \text { in } \Omega \\
    v |_{\partial \Omega} = g_2 & \text { on } \partial\Omega.
  \end{aligned}\right. 
\end{equation}
for the classical Laplace operator $\Delta$, and an integral-equation
problem for a certain weakly-singular integral operator.  The Poisson
problem~\eqref{eq:poisson} with, say,
$g_1\in C^\delta(\overline{\Omega})$ ($0<\delta\leq 1$) and
$g_2\in C(\partial\Omega)$, may be tackled~\cite[Thms. 2.28 and
3.40]{folland2020introduction} by means of the single- and
double-layer surface potentials and volumetric potentials defined in
terms of the fundamental solution
\begin{equation}
  N(x,y) =
  \begin{cases}
        \frac{1}{\omega_n}\frac{|x-y|^{2-n}}{(2-n)},& n>2,\\
        \frac{1}{2\pi} \log|x-y|,& n = 2
  \end{cases}
\end{equation}
of the Laplace operator, where $\omega_n$ denotes the surface area of
the $n$-dimensional sphere. In detail, the volumetric potential
\begin{equation}\label{single_layer_vol}
  V[f](x) = \int_{\Omega} N(x,y) f(y) dy, \quad x \in \overline{\Omega},
\end{equation}
with, say, $f\in C^\delta(\overline\Omega)$ with $0< \delta\leq 1$,
satisfies the equation~\cite[Thm. 2.28]{folland2020introduction}
\begin{equation}
  \label{eq:poisson_solve}
  \Delta V[f] = f.
\end{equation}
Thus,  defining
\begin{equation}
  \label{eq:v-poten}
  \widetilde v(x) = v(x) -  V[g_1](x)
\end{equation}
the problem~\eqref{eq:poisson} reduces to the Laplace problem
\begin{equation}\label{eq:laplace}
  \left\{\begin{aligned}
  &   \Delta \widetilde v = 0 &&\text{in } \Omega \\
  &  \widetilde v\big|_{\partial \Omega} = h  &&\text {on } \partial\Omega
  \end{aligned}\right. \quad ,
\quad \mbox{with}\quad h = g_2 -  V[g_1]\big|_{\partial \Omega} \in C(\partial \Omega).
\end{equation}

To solve~\eqref{eq:laplace}, in turn, we introduce, for
$\beta,\zeta\in C(\partial \Omega)$, the single- and double-layer
potentials~\cite{folland2020introduction}
\begin{equation}\label{single_layer_dom}
  \mathscr{S}[\beta](x) =  \int_{\partial\Omega} N(x,y) \beta(y) dS_y, \quad x \in \Omega, 
\end{equation}
and
\begin{equation}\label{double_layer_dom}
  \mathscr{D}[\zeta](x) =\int_{\partial\Omega}\partial_{\nu_y}N(x,y) \zeta(y) dS_y, \quad x \in \Omega,
\end{equation}
where
\begin{equation}
  \label{eq:normal}
  \mbox{$\nu_y$ denotes the outer normal vector to $\partial\Omega$
    at the point $y$.}
\end{equation}
For $\beta,\zeta\in C(\partial \Omega)$, the boundary values of
$\mathscr{S}[\beta](x)$ and $\mathscr{D}[\zeta](z) $ as
$z\to x\in \partial\Omega$ from the interior of $\Omega$ are given
by~\cite{folland2020introduction}
\begin{equation}\label{single_layer_bd}
  \lim_{z\to x,\; z\in\Omega}
  \mathscr{S}[\beta](z) = S[\beta](x), \quad x \in \partial\Omega
\end{equation}
and
\begin{equation}\label{double_layer_bd}
  \lim_{z\to x,\; z\in\Omega}
  \mathscr{D}[\zeta](z) = \frac{1}{2}\zeta(z)+ D[\zeta](x), \quad x \in \partial\Omega,
\end{equation}
where
\begin{equation}
S[\beta](x)  =\int_{\partial\Omega} N(x,y) \beta(y) dS_y\quad\text{and}\quad  D[\zeta](x) =\int_{\partial\Omega}\partial_{\nu_y}N(x,y) \zeta(y) dS_y, \quad x \in \partial\Omega
\end{equation}
denote the single- and double-layer operators, respectively. Note that
the kernel of the operator $D$, which is weakly singular for the
smooth domains considered in this paper, is given by
\begin{equation}
  \label{eq:DL-kernel}
  \partial_{\nu_y}N(x,y) = -\frac{1}{w_n}\frac{(x-y)\cdot
    \nu_y}{|x-y|^n}.
\end{equation}

As is known~\cite{folland2020introduction}, the solution of the
Laplace problem~\eqref{eq:laplace} for the (possibly
multiply-connected) smooth domains $\Omega$ under consideration may
indeed be obtained on the basis of the single- and double-layer
potentials just defined. Indeed, under the standing assumption of a
$C^\infty$ boundary $\partial\Omega$,
reference~\cite[Ch. 3]{folland2020introduction} tells us that, letting
$\beta_j$ ($j=1,\dots,n_h$) denote the (unique) solution of the
problem
\begin{equation}\label{eq:betaj_def}
    S[\beta_j](x) = \begin{cases}
        1, \ &\text{if } x \in \partial\Omega_{j}', \\
        0, \ &\text{if } x \in \partial\Omega_{\ell}',\ \ell =0,\dots,n_h , \ \ell\ne j
    \end{cases}
    \qquad \mbox{for}
   \qquad  x\in \partial \Omega,
 \end{equation}
 there exists a function $\zeta \in C(\partial \Omega)$ and a complex
 vector $a\in\mathbb{C}^{n_h}$ such that
\begin{equation}
  \label{eq:LP_to_Inteq}
  \widetilde v(x) = \mathscr{D}[\zeta](x) + \sum_{j=1}^{n_h} a_j \mathscr{S}[\beta_j](x)
\end{equation}
equals the solution of the problem~\eqref{eq:laplace}. In fact, any
couple $(\zeta,a)$ that solves the equation
\begin{equation} \label{eq:inteq_lap} \frac{1}{2}\zeta(x)+D[\zeta](x)
  + \sum_{j=1}^{n_h} a_j S[\beta_j](x)=h(x), \quad x\in
  \partial\Omega
\end{equation}
produces the solution $\widetilde v(x)$ upon substitution
into~\eqref{eq:LP_to_Inteq}. Moreover, for any given function $h$, the constant vector $a = (a_1,\dots,a_{n_h})\in \mathbb{C}^{n_h}$ is uniquely 
determined. This follows from the proof of point (a) in~\cite[Thm.~3.40]{folland2020introduction}, 
together with the result in~\cite[Cor.~3.39]{folland2020introduction}, which 
establishes that any $f \in L^2(\Omega)$ admits a unique representation of the form  
\[
  f = \widetilde{f} + \sum_{j=1}^{n_h} a_j \chi_{\partial \Omega_j}(x),
\]
with $\widetilde{f} \in
\mathrm{Image}\!\left(\tfrac{1}{2}I+D\right)$. Finally,
per~\cite[Prop. 3.37]{folland2020introduction}, the solution
of~\eqref{eq:inteq_lap} is unique if and only if $n_h=0$.


\section{Weakly Singular Fractional Laplacian Formulation \label{sec:Weakly_sing_form}}
As indicated in Section~\ref{intro}, the main results of this paper,
Theorems~\ref{thm:FL_into_Lap} through~\ref{thm:new_form_1D},
establish a connection between the FL boundary-value
problem~\eqref{frac_lapl} and a problem involving the composition of
the classical Laplace operator $\Delta_x$ with a weakly singular
integral operator over the domain~$\Omega$. In particular,
Theorem~\ref{thm:FL_into_Lap} presents the key re-formulation of the
FL operator, while Theorems~\ref{thm:new_form_nD}
and~\ref{thm:new_form_1D} (which concern the FL boundary-value problem
in dimensions $n\geq 2$ and $n = 1$, respectively) build upon this
reformulation to express the FL solutions in terms of a combined
volumetric-and-surface weakly singular integral equation over the
bounded domain~$\Omega$ and its boundary~$\partial\Omega$.

\begin{thm}\label{thm:FL_into_Lap}
  Let $n\in\mathbb{N}$, $n\ge 2$, $s \in (0,1)$, and let
  $u\in C^2(\Omega) \cap C(\mathbb{R}^n)$ which vanishes outside
  $\Omega$ be such that $u_{x_i}$ is integrable in a Lipschitz domain $\Omega$ for each
  $i$, $1 \le i \le n$. Then
  \begin{equation}\label{eq:frac-from-clas}
    (-\Delta)^{s} u(x) = C_{n,s} \Delta_x \int_{\Omega} \frac{u(y)}{|x-y|^{n+2s-2}} dy
  \end{equation}
  for $x \in \mathbb{R}^n \setminus \partial \Omega$, where
  \begin{equation}\label{eq:new_const_cns}
    C_{n,s} =
    -\frac{\Gamma\left(s+\frac{n}{2}-1\right)\Gamma(s)}{4^{1-s}\pi^{\frac{n}{2}+1}}\sin{(\pi
      s)}
  \end{equation}
  (cf. $c_{n,s}$ in eq.~\eqref{eq:frac_op}).

\end{thm}
\begin{thm}\label{thm:new_form_nD}
  Let $m,n\in\mathbb{N}$, $n, m \ge 2$, $s \in (0,1)$, and let
  $\Omega\subset\mathbb{R}^n$ denote a domain with an infinitely
  differentiable boundary and $n_h\geq 0$ holes, as detailed in
  Section~\ref{prelim}, and let $f\in
  C^{m+0}(\overline{\Omega})$. Further, let
  $F_s:C(\overline{\Omega})\to C(\overline{\Omega})$ denote the
  operator given by
  \begin{equation}
    \label{eq:F_s-op}
    F_s[\phi](x) = \int_{\Omega} \frac{d^s(y)\phi(y)}{|x-y|^{n+2s-2}}
    dy, \quad x \in \overline{\Omega},
  \end{equation}
  where $d=d(x)$ is defined prior to Theorem~\ref{thm:fu_rel} and where
  $C_{n,s}$ is given by equation~\eqref{eq:new_const_cns}. Then,
  employing the operators $\mathscr{D}$, $\mathscr{S}$, $D$, $S$ and
  $V$ introduced in Section~\ref{prelim} and
  letting $\beta_j$ denote the solution of
  equation~\eqref{eq:betaj_def} ($j=1,\dots,n_h$), the weakly-singular system
  of equations
\begin{empheq}[left={\empheqlbrace}]{alignat=2}
    &\mathscr{D}[\zeta](x) + \sum_{j=1}^{n_h} a_j \mathscr{S}[\beta_j](x)- C_{n,s}F_s[\phi](x) =  -V[f](x), \quad x\in \Omega,
 \label{eq:reformul1}   \\
    &\frac{1}{2}\zeta(x)+D[\zeta](x) + \sum_{j=1}^{n_h} a_j S[\beta_j](x)- C_{n,s}F_s[\phi](x) =  -V[f](x), \quad x\in \partial\Omega.\label{eq:reformul2}
  \end{empheq}
  admits a solution
  \begin{equation}\label{eq:uniq_triple}
    (\phi,\zeta, a)\in  C^{m+s-0}(\overline{\Omega}) \cap C^{m+2s-0}(\Omega) \times C^{m+1-0}(\partial \Omega) \times \mathbb{C}^{n_h}.
  \end{equation}
  Further, the elements $\phi$ and $a$ of the solution vector are
  uniquely determined among all possible pairs
  $(\phi,a)\in C^2(\overline{\Omega})\times \mathbb{C}^{n_h}$, while
  the element $\zeta \in C(\partial\Omega)$ is uniquely determined if and only if
  $n_h=0$. Finally, the function
\begin{equation}
    \label{eq:FL-sol-ws}
    u(x) =\begin{cases} d^s(x)\phi(x) & \text{for } x\in \Omega, \\ 0 &  \text{for } x\in \mathbb{R}^n \setminus \Omega, \end{cases}
  \end{equation}
  is the solution of the boundary value problem~\eqref{frac_lapl}.
\end{thm}

A 1D version of Theorem~\ref{thm:new_form_nD} takes a slightly
different form, as detailed in Theorem~\ref{thm:new_form_1D} below.
\begin{thm}\label{thm:new_form_1D}
  Let $m\in\mathbb{N}$, $ m \ge 2$, $n=1$, $s \in (0,1)$,
  $\Omega = (a,b)$, $f \in C^{m+0}[a, b]$,
\[
C_{s}= \begin{cases} -\frac{1}{\pi}\Gamma(2 s-1) \sin (\pi s)  , &s \neq \frac{1}{2}\\ \frac{1}{\pi} , &s=\frac{1}{2} \end{cases}, \qquad\qquad \text{and} \qquad\qquad  d(y) = (y-a)(b-y),
\]
and call $F_s:C(\overline{\Omega})\to C(\overline{\Omega})$ the operator 
given by
\[
  F_s[\phi](x) = \begin{dcases}  \int_{\Omega} |x-y|^{1-2s} d^s(y)\phi(y) dy, &s\ne\frac{1}{2}\\
    \int_{\Omega} \log(|x-y|) d^s(y)\phi(y) dy
                                                                              &s=\frac{1}{2}.\end{dcases}
\]
Then, letting $P(x)$ denote the double primitive of $f$, $P''(x) = f$, the weakly-singular equation
\begin{equation}\label{eq:new_form_1D}
  C_sF_s[\phi](x)  = P(x)-\zeta_1 x-\zeta_2, \quad x \in \overline{\Omega}
\end{equation}
for the unknowns 
\begin{equation}
  \label{eq:sol-1D}
  (\phi,\zeta_1,\zeta_2) \in  C^{m+s-0}(\overline{\Omega}) \cap C^{m+2s-0}(\Omega) \times \mathbb{C} \times \mathbb{C}
\end{equation}
(where, in particular $\zeta_1$ and $\zeta_2$ are complex constants)
admits a unique solution $(\phi,\zeta_1,\zeta_2)$. Finally, the
function $u$ given by~\eqref{eq:FL-sol-ws} with $\phi$ equal to the
first component of the solution vector~\eqref{eq:sol-1D} is the
solution of the boundary value problem~\eqref{frac_lapl} with $n=1$.
\end{thm}

The proofs of Theorems~\ref{thm:FL_into_Lap} through~\ref{thm:new_form_1D}
are given at the end of this section, following a sequence of four
preparatory lemmas. The first two lemmas re-express the FL
operator~\eqref{eq:frac_op} in terms of the classical Laplace operator
and a volumetric weakly singular integral operator. The final two
lemmas, in turn, concern the regularity of the boundary restrictions
$\left . F_s[\phi]\right|_{\partial\Omega}$ and
$\left . V[f]\right|_{\partial\Omega}$, respectively. Throughout
this paper the $n$-dimensional open ball centered at $x$ of radius $a$
is denoted by
\begin{equation}
  \label{eq:ball}
  B_a(x)=\{y\in\mathbb{R}^n\ :\ |y-x|<a \}.
\end{equation}

\begin{lem}\label{lem:FL_op_grad}
  Under the hypothesis of Theorem~\ref{thm:FL_into_Lap}, the integral
  $\int_{\Omega} |x-y|^{-n-2s+2}\frac{\partial u}{\partial y_i}(y) dy$
  is a differentiable function of $x_i$ ($x \in \Omega$,
  $i=1,\cdots,n$), and we have
\begin{equation} \label{eq:FL_op_grad}
(-\Delta)^{s} u(x)=C_{n,s} \nabla_{x} \cdot \int_{\Omega} \frac{1}{|x-y|^{n+2s-2}}\nabla u(y) dy, \quad x \in \mathbb{R}^n \setminus \partial \Omega.
\end{equation}
\end{lem}
\begin{proof}
 Using a linear change of
  integration variables, equation~\eqref{eq:frac_op} is re-expressed in
  the form
  \begin{equation}
      (-\Delta)^{s} u(x)=c_{n,s} P.V\int_{\mathbb{R}^n}
    (u(x)-u(x-y))|y|^{-n-2s}dy.
  \end{equation}
\noindent
Noting that the support of $u(x-y)$ as a function of $y$ is contained
in the set $x-\Omega = \left\{ x-w : w \in \Omega \right\}$ and using
the decomposition
$\mathbb{R}^n=(x-\Omega) \cup \left(\mathbb{R}^n\setminus
  (x-\Omega)\right)$ we obtain
\begin{equation}\label{eq:FL_dom_decomp}
(-\Delta)^{s} u(x) = c_{n,s} \left\{ P.V\int_{x-\Omega} (u(x)-u(x-y))|y|^{-n-2s}dy + \int_{\mathbb{R}^n\setminus(x-\Omega)}u(x)|y|^{-n-2s}dy \right\}. 
\end{equation}
(The Cauchy principal value integral is not included in the second
term in \eqref{eq:FL_dom_decomp} as the integrand is smooth for all
$x\in\mathbb{R}^n$, since $u(x)=0$ for $x \notin \Omega$, and
$0 \notin \mathbb{R}^n\setminus(x-\Omega)$ for $x \in \Omega$.)

We first establish the lemma in the case $x \notin \Omega$, or
$0 \notin x-\Omega$. In this case the integrand in the first integral
in~\eqref{eq:FL_dom_decomp} is smooth throughout the domain of
integration and the second integral vanishes; we thus obtain
\begin{equation}
(-\Delta)^{s} u(x)= -c_{n,s}\int_{x-\Omega} u(x-y)|y|^{-n-2s}dy.
\end{equation}
Since $|y|^{-n-2s}=\frac{1}{2s(2s+n-2)}\Delta \frac{1}{|y|^{n+2s-2}}$,
and since $u\in C^2(\Omega)\cap C(\overline{\Omega})$ and vanishes
on $\partial \Omega$, integration by parts over the Lipschitz integration domain~\cite[Thm. 3.34]{mclean2000} yields
\begin{equation}\label{eq:FL_xout_Intbyparts}
    (-\Delta)^{s} u(x) = -\frac{c_{n,s}}{2s(2s+n-2)} \int_{x-\Omega} (\nabla u)(x-y)\nabla \frac{1}{|y|^{n+2s-2}}dy.
\end{equation}
Letting $C_{n,s} = -\frac{c_{n,s}}{2s(2s+n-2)}$ and using the
change of variables $z = x - y$, \eqref{eq:FL_xout_Intbyparts} becomes
\begin{equation}\label{int-part1}
  (-\Delta)^{s} u(x) = C_{n,s} \int_{\Omega}\nabla u(z)\nabla_{x} \frac{1}{|x-z|^{n+2s-2}}dz.
\end{equation}
Clearly, in the present case $x \notin \Omega$ the integrand
in~\eqref{int-part1} is a continuously differentiable function of $x$,
and, thus, the gradient in $x$ can be moved outside of the integral
sign, which establishes~\eqref{eq:FL_op_grad} for $x \notin
\Omega$.

To consider the case $x \in \Omega$, in turn, we let
$D_{\varepsilon} = (x-\Omega)\setminus B_{\varepsilon}(0)$ (see
eq.~\eqref{eq:ball}), and we call $\widehat{\xi}_1$ the outer
normal to $x-\Omega$ and $\widehat{\xi}_2$ the inner normal to
$\partial B_{\varepsilon}(0)$. Using integration by parts the first
term in \eqref{eq:FL_dom_decomp} becomes
\begin{gather}
c_{n,s}P.V\int_{x-\Omega} (u(x)-u(x-y))|y|^{-n-2s}dy =-C_{n,s}\lim_{ \varepsilon \to 0^{+}} \int_{D_{\varepsilon}}(u(x)-u(x-y))\Delta |y|^{-n-2s+2}dy, \label{eq:main_int}\\
 = \lim_{ \varepsilon \to 0^{+}}\left[C_{n,s}\int_{D_{\varepsilon}}\nabla_{y}(u(x)-u(x-y))\cdot \nabla_{y} |y|^{-n-2s+2}dy\right.  \\ 
-C_{n,s}\int_{x-\partial\Omega}\left(\nabla \frac{1}{|y|^{n+2s-2}}\cdot \widehat{\xi}_1\right)\left(u(x)-u(x-y)\right)dS_{y} \label{eq:outer_bd}\\
\left. -C_{n,s} \int_{\partial B_{\varepsilon}(0)}\left(\nabla \frac{1}{|y|^{n+2s-2}}\cdot \widehat{\xi}_2\right)\left(u(x)-u(x-y)\right)dS_{y}\right]. \label{eq:inner_bd}
\end{gather}
The integral in~\eqref{eq:outer_bd} is independent of $\varepsilon$
and will subsequently cancel out the second term
in~\eqref{eq:FL_dom_decomp}. Indeed, using integration by parts we
obtain
\begin{equation}
c_{n,s}\int_{\mathbb{R}^n\setminus(x-\Omega)}\!\! u(x)|y|^{-n-2s}dy = -C_{n,s} \int_{\mathbb{R}^n\setminus(x-\Omega)} \!\!\!\! u(x)\Delta |y|^{-n-2s+2}dy=C_{n,s}\int_{x-\partial\Omega}\!\! \left(\nabla \frac{1}{|y|^{n+2s-2}}\cdot \widehat{\xi}_1\right)u(x)dS_{y}
\end{equation}
which coincides with the negative of~\eqref{eq:FL_dom_decomp} since
$u(x-y)$ vanishes for $y\in x-\partial\Omega$. We next simplify the
expression \eqref{eq:inner_bd} and show that it vanishes in the limit
as $\varepsilon$ tends to zero. Indeed, for the gradient
in~\eqref{eq:FL_dom_decomp} we have
\begin{equation}
    \nabla \frac{1}{|y|^{n+2s-2}}\cdot \widehat{\xi}_2 = \left(-(n+2s-2)\frac{y}{|y|^{n+2s}}\right)\cdot \left(-\frac{y}{\varepsilon}\right)=\frac{n+2s-2}{\varepsilon^{n+2s-1}}.
\end{equation}
Further, parametrizing $\partial B_{\varepsilon}(0)$ using
$n$-dimensional spherical coordinates $y = \varepsilon \xi$,
\begin{equation}\label{spher_1}
  \begin{aligned}
    &\xi_1 =  \cos{\phi_1}, \quad \xi_j = \cos{\phi_j} \prod_{k=1}^{j-1} \sin{\phi_k}, \quad (j=2,\cdots,n-2) \\
    &\xi_{n-1} = \cos{\theta} \prod_{k=1}^{n-2} \sin{\phi_k}, \quad \xi_n =  \sin{\theta} \prod_{k=1}^{n-2} \sin{\phi_k},
  \end{aligned}
  \end{equation}
($0 \le \phi_j \le \pi$, $0 \le \theta < 2\pi$), letting
$\phi = (\phi_1,\dots,\phi_{n-2})$ and defining~\cite{folland1999}
\begin{equation}\label{spher_elems}
\xi(\phi,\theta) = (\xi_1,\dots, \xi_n),\qquad  F(\phi) = \prod_{k=1}^ {n-2} (\sin{\phi_k})^{n-k-1}, \qquad dS_{y}=\varepsilon^{n-1}F(\phi) d\phi d\theta,
\end{equation}
the integral~\eqref{eq:inner_bd} becomes
\begin{equation}\label{eq:int_spherical_coord}
-C_{n,s}\frac{n+2s-2}{\varepsilon^{2s}} \left[ \int_{[0,\pi]^{n-2} \times [0,\pi]}+\int_{[0,\pi]^{n-2}\times [\pi,2\pi]} \left(u(x)-u\left(x-\varepsilon \xi(\phi,\theta)\right)\right)F(\phi)  d\phi d\theta \right].
\end{equation}
Now, introducing the changes of variables $\theta \mapsto \theta+\pi$
and $\phi_j \mapsto \pi-\phi_j$ ($1\leq j\leq n-2$) in the second
integral in~\eqref{eq:int_spherical_coord} and using Taylor's theorem,
$u(x\pm h) = u(x) \pm \nabla u(x) \cdot h +
\frac{1}{2}h^T Hu(\mu^\pm)h$ (where $H$ denotes the
Hessian matrix and $\mu^\pm_i$ is between $x_i$ and $x_i\pm h_i$), the
expression~\eqref{eq:int_spherical_coord} becomes
\begin{gather*}
    -C_{n,s}\frac{n+2s-2}{\varepsilon^{2s}} \left[ \int_{[0,\pi]^{n-2} \times [0,\pi]}(2u(x)-u(x-\varepsilon \xi(\phi,\theta))-u(x+\varepsilon \xi(\phi,\theta)))F(\phi) d\phi d\theta  \right] \\
    = -C_{n,s}\frac{n+2s-2}{\varepsilon^{2s}} \left[ \int_{[0,\pi]^{n-2} \times [0,\pi]} \frac{\varepsilon^2}{2}\left( \xi(\phi,\theta)^T(Hu(\mu^+)+Hu(\mu^-))\xi(\phi,\theta)\right)F(\phi)  d\phi d\theta \right].
\end{gather*}
Since $s<1$ we see that the term \eqref{eq:inner_bd} tends to zero
as $\varepsilon \to 0^+$.

In all, the discussion above tells us that
\begin{align}
(-\Delta)^{s} u(x) &= C_{n,s}\lim_{\varepsilon \to 0}\int_{D_{\varepsilon}} (\nabla u)(x-y) \cdot \nabla |y|^{-n-2s+2}dy \nonumber \\
                   &= C_{n,s} P.V \int_{x-\Omega} (\nabla u)(x-y) \cdot \nabla |y|^{-n-2s+2}dy \nonumber \\
  &= C_{n,s} P.V \int_{\Omega} \nabla u(y) \cdot \nabla_{x} |x-y|^{-n-2s+2}dy. \label{eq:2nd_case_int}
\end{align}
Since, by hypothesis, the function
$\frac{\partial u}{\partial y_j}(y)$ is integrable over $\Omega$ for
each $j=1,\cdots,n$, Lemma~\ref{lem:und_int} below may be applied to
$w=\frac{\partial u}{\partial y_j}(y)$ to show that the $x$-gradient
in~\eqref{eq:2nd_case_int} can be moved outside the Principal Value
integral, and, thus,
\begin{equation}\label{eq:grad_PV_inter_2}
(-\Delta)^{s} u(x)  =  C_{n,s}\nabla_{x} \cdot \int_{\Omega}\frac{1}{|x-y|^{n+2s-2}}\nabla u(y) dy, 
\end{equation}
as desired. The proof is now complete.
\end{proof}
In what follows $L^1(\Omega)$ denotes the set of all integrable
functions defined in the set $\Omega$.

\begin{lem}\label{lem:und_int}
  Let $s\in (0,1)$, $n\in\mathbb{N}$, $n \ge 2$, and let
  $w \in C^{1}(\Omega)\cap L^1(\Omega)$ where
  $\Omega\subset\mathbb{R}^n$ is an open bounded Lipschitz
  domain. Then for $i=1,\cdots,n$ the integral
  $\int_{\Omega}w(y)|x-y|^{-n-2s+2} dy$ is differentiable with respect
  to $x_i$ at each $x \in \Omega$, and we have
 \begin{equation}
   \frac{\partial }{\partial x_i}\int_{\Omega}w(y)|x-y|^{-n-2s+2} dy 
   = P.V \int_{\Omega} w(y) \frac{\partial }{\partial x_i} |x-y|^{-n-2s+2}dy,\quad x \in \Omega.  \label{eq:PV_grad_inter}
\end{equation}
\end{lem}
\begin{proof}
    See appendix \ref{Ap1}.
\end{proof}

The following lemma plays an important role in the regularity proof
for the double-layer unknown $\zeta$ in
equations~\eqref{eq:reformul1}-\eqref{eq:reformul2}.
\begin{lem}\label{lem:boundary_reg_Fs}
  Let $n \in \mathbb{N}$ with $n \geq 2$, and let
  $\Omega \subset \mathbb{R}^n$ be a bounded open domain with an
  infinitely differentiable boundary. Suppose the pair $(s,n)$
  satisfies either $s \in (0,1)$ with $n = 2$, or $s \in [0,1)$ with
  $n \geq 3$. Then, for any non-negative integer $m$ and any function
  $\phi \in C^m(\overline{\Omega})$, the boundary trace
  $F_s[\phi]\big|_{\partial\Omega}$ of $F_s[\phi]$
  (eq.~\ref{eq:F_s-op}) satisfies
\[
F_s[\phi]\big|_{\partial \Omega} \in C^{m+1}(\partial \Omega).
\]
\end{lem}
\begin{proof}
  The proof is divided into several sections for clarity.
\paragraph{Tubular Neighborhood.} For a given $z \in \partial \Omega$, letting
  $\gamma:[0,1]^{n-1}\to \partial\Omega$ denote a smooth
  parametrization of a neighborhood of $z$ in $\partial \Omega$,
  using~\eqref{eq:normal}, taking a sufficiently small
  $\varepsilon >0$, and using the variable
  $t=(t_1,t_2,\dots,t_n)\in[0,1]^{n-1}\times [0,\varepsilon]$, we
  construct the $C^\infty$ parametrization
  \begin{equation}
    \label{eq:neigh_param}
    r^{z}: [0,1]^{n-1}\times [0,\varepsilon]\to \overline\Omega\quad\mbox{given by}\quad
    r^{z}(t) = \gamma(t_1,\cdots,t_{n-1})-t_n \nu(t_1,\cdots,t_{n-1}),
  \end{equation}
  of a neighborhood of $z$
  in the relative topology of $\overline\Omega$.  In what follows we
  let
  \begin{equation}
    \label{eq:R-V_def}
   R_{\varepsilon} = [0,1]^{n-1} \times [0,\varepsilon], \quad V_{\varepsilon}^z = r^{z}\big(R_{\varepsilon}\big),\quad \mbox{and}\quad S^{z} = r^{z}\big((0,1)^{n-1} \times\{0\}\big),
  \end{equation}
  and we show that $F_s[\phi]\big|_{S^{z}}\in C^{m+1}(S^{z}) $.  The
  set $V_{\varepsilon}^z$ is a
  ``tubular neighborhood'' of $z$ within $\overline \Omega$ and along
  $\partial\Omega$~\cite{folland2020introduction}.  For notational simplicity we drop the subindex
  $z$ in the parametrization just introduced, and we write
  \begin{equation}
    \label{eq:drop}
    r^{z}(t)=r(t)=(r_1(t),r_2(t),\dots,r_n(t)) \quad \mbox{for}\quad t\in R_{\varepsilon}.
  \end{equation}
\begin{figure}
    \centering
    \includegraphics[scale=1]{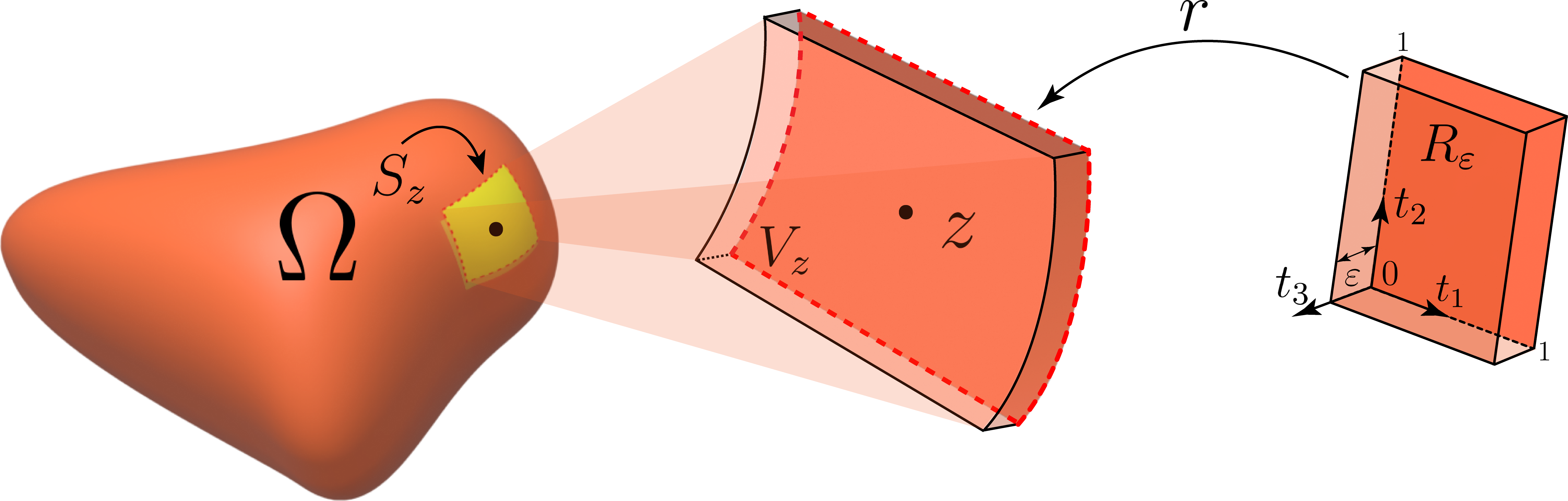}
    \caption{Tubular neighborhood for the domain $\Omega$ in dimension
      $n=3$, around the point $z$, and parametrized over the
      rectangular domain
      $R_\varepsilon = [0,1]^{2}\times [0,\varepsilon]$.\label{fig:tubular}}
\end{figure}
  Letting $J$ denote the Jacobian matrix 
  \begin{equation}\label{eq:AP2_jacb_r}
    J_{k\ell} = \frac{\partial r_k}{\partial y_{\ell}}
  \end{equation}
  of the parametrization $r=r(t)$, and denoting by $|J(t)|$ the
  absolute value of the corresponding Jacobian determinant, the
  integral~\eqref{eq:F_s-op} may be re-expressed in the form
\begin{equation}\label{eq:AP2_Int1}
    F_s[\phi](x) = \int_{\Omega} \frac{d^s(y)\phi(y)}{|x-y|^{n+2s-2}}dy = \int_{R_{\varepsilon}} \frac{d^s(r(t))\phi(r(t))}{|x-r(t)|^{n+2s-2}}|J(t)|dt + \int_{\Omega \setminus V_{\varepsilon}^z} \frac{d^s(y)\phi(y)}{|x-y|^{n+2s-2}}dy. 
\end{equation}
Clearly, the second integral on the right-hand side
of~\eqref{eq:AP2_Int1} is infinitely differentiable at and around $z$
and it therefore suffices to show that the restriction of the first
right-hand integral to $\partial \Omega$ is a $C^{m+1}$ function of
$x$ for $x\in S^z$.

To do this we define the function $\varphi(t) = \phi(r(t)) |J(t)|$,
$\varphi\in C^m(R_{\varepsilon})$, and we note that, taking, as we
may, $\varepsilon$ sufficiently small, in view
of~\eqref{eq:neigh_param} we have $d(r(t)) = t_{n}$ for
$t\in R_{\varepsilon}$. Thus, defining
\begin{equation}
  \label{eq:tau}
  \tau = (\tau_1,\dots,\tau_{n-1}),\quad\mbox{writing}\quad (\tau,0) = (\tau_1,\dots,\tau_{n-1},0)\quad\mbox{when needed},
\end{equation}
and calling $\widetilde F (\tau)$ the first right-hand integral
in~\eqref{eq:AP2_Int1} with $x$ substituted by $r(\tau,0)$, we have
\begin{equation}\label{eq:AP2_Int2}
     \widetilde F(\tau)  = \int_{R_{\varepsilon}} \frac{(t_n)^s\varphi(t)}{|r(\tau,0)-r(t)|^{n+2s-2}}dt, \quad \tau \in (0,1)^{n-1}.
\end{equation} 
Clearly, to complete the proof it suffices to establish the following statement:
\begin{equation}\label{gen_stat}
  \mbox{``If $\varphi\in C^m(R_{\varepsilon})$, then, the integral $\widetilde F$ in~\eqref{eq:AP2_Int2} is an element of
    $C^{m+1}((0,1)^{n-1})$,''}
\end{equation}
This result is proved by induction on $m$, following a sequence of
preparatory steps.
\paragraph{Geometric Inequalities.}
The proof relies on the elementary  inequalities
\begin{equation}\label{eq:rh_bounds}
|r(\tau,0)-r(t)| \ge C_1|(\tau,0)-t| \quad\mbox{and}\quad  | h_{\alpha,\beta}(\tau,t)| \le C_2|(\tau,0)-t|^2,
\end{equation}
which hold for certain constants $C_1,C_2>0$. Here, for
$\alpha , \, \beta \in \mathbb{N}_0^{n-1}$, $t \in R_{\varepsilon}$
and $\tau \in (0,1)^{n-1}$, we have set
\begin{equation}\label{eq:AP2_defn_cq_h2}
   h_{\alpha,\beta}(\tau,t) = (D^{\alpha}_\tau r(\tau,0) - D^{\alpha}_t r(t))\cdot(D^{\beta}_\tau r(\tau,0) - D^{\beta}_t r(t)),
\end{equation}
see~\eqref{eq:AP2_defn_Dalp} and accompanying discussion.  These
inequalities can be established on the basis of the identities
\begin{equation}\label{eq:int_der}
  r(\xi)-r(\eta) = \int_0^1 \frac{d}{d\mu}\left[r(\mu(\xi-\eta)+\eta) \right] d\mu =\left( \int_0^1 J(\mu(\xi-\eta)+\eta)d\mu\right) (\xi-\eta)^T = A(\xi-\eta)^T,
\end{equation}
where $J$ and $A$ denote the Jacobian matrix~\eqref{eq:AP2_jacb_r} and
the integral containing the Jacobian in~\eqref{eq:int_der},
respectively. Indeed, the first bound in~\eqref{eq:rh_bounds} follows
directly from the inequality
$| A(\xi-\eta)^T|\geq |\xi-\eta|/\| A^{-1}\|_2$ that results from the
relation $(\xi-\eta)^T = A^{-1}A(\xi-\eta)^T$, where $\|\cdot \|_2$ denotes
the matrix norm induced by the Euclidean norm in $\mathbb{R}^n$. The
second bound in~\eqref{eq:rh_bounds}, in turn, follows from two
applications of~\eqref{eq:int_der}, one with $r$ substituted by
$D^\alpha r$ and the other with $r$ substituted by $D^\beta r$.
\paragraph{Simple Motivating Discussion.} To motivate the proof
technique, we begin with a brief discussion
establishing~\eqref{gen_stat} for the particular cases $m = 0$ and
$m = 1$. To handle the case $m = 0$ we first note that, for $1 \leq j \leq n-1$
and for fixed values of
$\tau_1 \dots, \tau_{j-1}, \tau_{j+1}, \dots, \tau_{n-1} \in (0,1)$,
the derivative of the integrand in~\eqref{eq:AP2_Int2} with respect to
$\tau_j$, which is given by
\begin{equation}\label{eq:AP2_firstint}
  H(\tau,t)= -(n+2s-2)\frac{(t_n)^s  (r(\tau,0)-r(t)) \cdot \frac{\partial r}{\partial \tau_j}(\tau,0)  }{|r(\tau,0)-r(t)|^{n+2s}}\varphi(t),
\end{equation}
is integrable with respect to $(\tau_j, t)$ in the domain $(0,1)\times R_{\varepsilon}$. Indeed, using the inequalities
\begin{equation}\label{eq:AP2_inequality_r}
\frac{|t_n|^s}{|r(\tau,0)-r(t)|^{n+2s-1}}
\le \frac{|(\tau_1,\dots,\tau_{n-1},0)-t|^s}{|r(\tau_1,\dots,\tau_{n-1},0)-r(t)|^{n+2s-1}}
\le \frac{\text{const.}}{|(\tau_1,\dots,\tau_{n-1},0)-t|^{n-(1-s)}},
\end{equation}
the second of which follows from~\eqref{eq:rh_bounds}, we see that
$H(\tau,t)$ is integrable with respect to $t$ for each fixed $\tau_j$,
and the resulting $t$–integral depends continuously on
$\tau_j$. Hence, $H(\tau,t)$ is integrable with respect to
$(\tau_j, t)$ in the domain $(0,1)\times R_{\varepsilon}$, as claimed.
In view of Lemma~\ref{lem:Ap_Fubini} and Remark~\ref{rem:continuity} we conclude that
\begin{equation}\label{eq:ideap0}
  \frac{\partial}{\partial \tau_j}\widetilde F(\tau) = -(n+2s-2)\int_{R_\varepsilon} \frac{(t_n)^s(r(\tau,0)-r(t))\cdot \frac{\partial}{\partial \tau_j}r(\tau,0)}{|r(\tau,0)-r(t)|^{n+2s}}\varphi(t) dt,
\end{equation}
and the derivative~\eqref{eq:ideap0}
is continuous with respect to $\tau$---that is, the instance $m=0$ of
statement~\eqref{gen_stat} has been verified, as desired.

Unfortunately, we cannot directly apply Lemma~\ref{lem:Ap_Fubini} to
differentiate~\eqref{eq:ideap0} a second time to
establish~\eqref{gen_stat} in the case $m=1$---since the corresponding
second derivative of the integrand is not integrable. To address this
difficulty, we add and subtract $\frac{\partial}{\partial t_j}r(t)$
from $\frac{\partial}{\partial \tau_j}r(\tau,0)$ in the integrand and
thus re-express the resulting integral as a sum of two integrals,
$I_1$ and $I_2$, where
\begin{equation}\label{eq:p0I1_def1}
    I_1 = -(n+2s-2)\int_{R_\varepsilon}  \frac{(t_n)^s (r(\tau,0)-r(t))\cdot \left(\frac{\partial}{\partial \tau_j}r(\tau,0)-\frac{\partial}{\partial t_j}r(t) \right)}{|r(\tau,0)-r(t)|^{n+2s}} \varphi(t) dt
\end{equation}
and
\begin{equation}\label{eq:p0I1_def2}
    I_2 = -(n+2s-2)\int_{R_\varepsilon} \frac{(t_n)^s (r(\tau,0)-r(t))\cdot\frac{\partial}{\partial t_j}r(t) }{|r(\tau,0)-r(t)|^{n+2s}}\varphi(t)  dt.
\end{equation}
It is easy to check that, owing in particular to the factor $(t_n)^s$ in
the numerator, the derivative of the integrand in
$I_1$~\eqref{eq:p0I1_def1} with respect to $\tau_j$
($1\leq j\leq (n-1)$) is integrable, and, thus, by
Lemma~\ref{lem:Ap_Fubini}, this integral is a continuously
differentiable function of $\tau$. Further, the fraction in the
integrand of $I_2$~\eqref{eq:p0I1_def2} equals the derivative with respect to $t_j$ of a
certain quotient, and, thus, using integration by parts, this integral
can be expressed in the form
\begin{equation}\label{eq:p0I2_def}
  I_2 =   \int_{R_\varepsilon} \frac{(t_n)^s }{|r(\tau,0)-r(t)|^{n+2s-2}}\frac{\partial\varphi (t)}{\partial t_j} dt -\int_{R_\varepsilon} \frac{\partial}{\partial t_j}\left(\frac{(t_n)^s \varphi(t)}{|r(\tau,0)-r(t)|^{n+2s-2}}\right) dt,\quad 1\leq j\leq (n-1).
\end{equation}
Clearly, the first integral in~\eqref{eq:p0I2_def} has the same form
as~\eqref{eq:AP2_Int2}, but with $\varphi$ replaced by its derivative
$\frac{\partial \varphi}{\partial t_j}(t)$.  Therefore, as established
for~\eqref{eq:AP2_Int2}, this integral is itself continuously
differentiable. Using Barrow's formula, in turn, the second integral
in~\eqref{eq:p0I2_def} is expressed as the sum of two
$(n-1)$-dimensional integrals over the boundary sets
$R_\varepsilon\cap\{t_j=0\}$ and $R_\varepsilon\cap\{t_j=1\}$ (see
Figure~\ref{fig:tubular}), and it is thus a smooth function of $\tau$
for $\tau\in (0,1)^{n-1}$. This proof thus shows that
$\widetilde F \in C^{2}((0,1)^{n-1})$, as desired.

\paragraph{Additional Notations.} To establish the result in the
general case $m\geq 0$ we first introduce additional
notations. For $q\in \mathbb{N}_{0}$, we set
\begin{equation}\label{eq:AP2_defn_cq_h1}
  c_q = n+2s+2(q-1).
\end{equation}
Further, for $p,q \in \mathbb{N}$ and $\kappa\in \mathbb{N}_0^{n-1}$
satisfying $q \leq p$ and $|\kappa|=p$
(eq.~\eqref{eq:mlt-indx-nrm}), and using the notation
\begin{equation}\label{eq:bold_alphabeta}
  (\boldsymbol{\alpha},\boldsymbol{\beta}) =
  \left(\alpha^{(1)},\dots,\alpha^{(q)},\beta^{(1)},\dots,\beta^{(q)}\right)
  \in \left(\mathbb{N}_0^{n-1}\right)^{2q},
\end{equation}
(with $\alpha^{(\ell)}\in\mathbb{N}_0^{n-1}$,
$\beta^{(\ell)}\in\mathbb{N}_0^{n-1}$ for $1\leq \ell\leq q$), we
define the index set
\begin{equation}
  \mathcal{I}_{p,q}(\kappa) = \left\{ (\boldsymbol{\alpha},\boldsymbol{\beta})  \in \left(\mathbb{N}_0^{n-1}\right)^{2q}: 
    \  \alpha^{(\ell)} \leq \kappa, \ \beta^{(\ell)} \leq \kappa ,\ \  \sum_{\ell=1}^{q}(|\alpha^{(\ell)}|+|\beta^{(\ell)}|) \leq p \right\},\quad q\in \mathbb{N},
\end{equation}
where, for $\alpha, \beta \in \mathbb{N}_0^{n-1}$, $\alpha \le \beta$
denotes the componentwise inequality $\alpha_i \le \beta_i$ for
$i = 1, \dots, n-1$.

\paragraph{Inductive Statement.} Using~\eqref{eq:AP2_Int2} together
with the ``Additional Notations'' introduced above, we
show by induction in $m$ that the following statement holds for all
$m\in\mathbb{N}_0$:

\parbox{0.9\linewidth}{\em ``For $\varphi \in C^m(R_\varepsilon)$ and for each $\kappa\in \mathbb{N}_0^{n-1}$,
    $|\kappa| = m$, and each $\gamma \in \mathbb{N}_0^{n-1}$ with
    $|\gamma|=1$, we have
    $D^{\kappa}_\tau\widetilde{F} \in C^{1}\left((0,1)^{n-1}\right)$,
  and the $\gamma$ derivative
  \begin{equation}
  \label{eq:jder}
  D^\gamma_\tau D^{\kappa}_\tau\widetilde{F},\quad\mbox{where}\quad  D^\gamma_\tau = \frac{\partial}{\partial \tau_j}\quad\mbox{for a certain}\quad j\quad\mbox{with}\quad  1\leq j\leq n-1,
\end{equation}
is given by the expression
\begin{align}
    D^{\kappa+\gamma}_\tau \widetilde{F}(\tau) &= \sum_{q=1}^{m} \sum_{\substack{|\eta| \leq m-q \\ \eta \leq \kappa}} \sum_{(\boldsymbol{\alpha},\boldsymbol{\beta})\in \mathcal{I}_{m,q}(\kappa)} C^{\kappa}_{_{q,\eta,\boldsymbol{\alpha},\boldsymbol{\beta}}}\int_{R_\varepsilon}  (t_n)^s   D^{\gamma}_\tau\left[\frac{ \prod_{k=1}^{q} h_{\alpha^{(k)},\beta^{(k)}}(\tau,t)}{|r(\tau,0)-r(t)|^{c_q}}\right]D^{\eta}_t\varphi(t) dt
    \label{eq:AP2_main_term_alp_gen} \\
    &\quad -c_0\int_{R_\varepsilon}  \frac{(t_n)^s(r(\tau,0)-r(t))\cdot D^{\gamma}_\tau r(\tau,0) }{|r(\tau,0)-r(t)|^{c_1}}D^{\kappa}_t\varphi(t) dt+ S(\tau; \kappa, \gamma), \label{eq:AP2_smooth_term_alp_gen}
\end{align}
wherein all integrands are integrable functions of $t$,
$C^{\kappa}_{q,\eta,\boldsymbol{\alpha},\boldsymbol{\beta}}$
are constants, and $S(\tau; \kappa, \gamma)$ denotes a $C^{\infty}$
function of $\tau$ for $\tau \in (0,1)^{n-1}$.''}

\noindent (For notational simplicity
equation~\eqref{eq:AP2_main_term_alp_gen}--\eqref{eq:AP2_smooth_term_alp_gen}
contains more $\eta$ terms than are strictly necessary; the constants
$C^{\kappa}_{_{q,\eta,\boldsymbol{\alpha},\boldsymbol{\beta}}}$
corresponding to such additional terms are equal to zero.) Clearly,
the proof of the lemma will be complete once this statement is shown
to hold for all $m \in \mathbb{N}_0$.

\paragraph{Proof for $m=0$.} The proof in this case was established as
part of the ``Simple Motivating Discussion'' presented above; note
that for $|\kappa| = m=0$ and $|\gamma|=1$ we have $S(\tau; \kappa, \gamma)=0$.

\paragraph{Preparation for the Inductive Step.} To prepare for the
inductive step we first establish that, for
$\varphi \in C^{m}(R_{\varepsilon})$, the integrands appearing in the
integrals~\eqref{eq:AP2_main_term_alp_gen}–\eqref{eq:AP2_smooth_term_alp_gen}
are integrable. Moreover, for $\varphi \in C^{m+1}(R_{\varepsilon})$,
we introduce alternative representations of these integrals for which
first-order differentiation under the integral sign result in
integrands that remain integrable.

We begin by considering the
integral~\eqref{eq:AP2_smooth_term_alp_gen}. As claimed, the
corresponding integrand is clearly integrable, in view
of~\eqref{eq:AP2_inequality_r}. To obtain the desired alternative
representation for $\varphi \in C^{m+1}(R_{\varepsilon})$ in this
case, we add and subtract $D^{\gamma}_t r(t)$ from
$D^{\gamma}_\tau r(\tau,0)$ in the integrand, thereby rewriting the
integral as the sum of two terms, $I_1$ and $I_2$. Here, $I_1$
contains the difference $h_{0,\gamma}(\tau,t)$
(see~\eqref{eq:AP2_defn_cq_h2}), while $I_2$ involves the term
$D^{\gamma}_t r(t)$:
\begin{equation}\label{eq:I1_I2_def}
    I_1 = -c_0\int_{R_\varepsilon}  \frac{(t_n)^s h_{0,\gamma}(\tau,t)}{|r(\tau,0)-r(t)|^{c_1}} D^{\kappa}_t\varphi(t) dt,\quad I_2 = -c_0\int_{R_\varepsilon}  \frac{(t_n)^s (r(\tau,0)-r(t))\cdot D^{\gamma}_tr(t)}{|r(\tau,0)-r(t)|^{c_1}}D^{\kappa}_t\varphi(t) dt.
\end{equation}
In view of~\eqref{eq:jder} (and~\eqref{eq:AP2_defn_Dalp}), and since
$j\ne n$, the integrand in $I_2$ equals the product of
$D^{\kappa}_t\varphi(t)$ and a derivative with respect to $t_j$ of the
kernel $(t_n)^s / |r(\tau,0)-r(t)|^{c_0}$,
\begin{equation}\label{eq:I2_der_ker}
    I_2 = -\int_{R_\varepsilon}  D^{\gamma}_t\left(\frac{(t_n)^s}{|r(\tau,0)-r(t)|^{c_0}}\right)D^{\kappa}_t\varphi(t) dt.
\end{equation}
Using the product differentiation rule to re-express the integrand
in~\eqref{eq:I2_der_ker} we obtain the desired alternative expression
\begin{equation}\label{eq:AP2_First_partial_eq}
I_1+I_2 =  -c_0\int_{R_\varepsilon} \frac{(t_n)^s  h_{0,\gamma}(\tau,t)D^{\kappa}_t\varphi(t)}{|r(\tau,0)-r(t)|^{c_1}} dt + \int_{R_\varepsilon} \frac{(t_n)^s D^{\kappa+\gamma}_t\varphi(t)}{|r(\tau,0)-r(t)|^{c_0}} dt -\int_{R_\varepsilon} D^{\gamma}_t\left(\frac{(t_n)^s D^{\kappa}_t\varphi(t)}{|r(\tau,0)-r(t)|^{c_0}}\right) dt
\end{equation}
for the integral~\eqref{eq:AP2_smooth_term_alp_gen}. Calling
$S^1(\tau;\kappa, \gamma)$ the third integral in~\eqref{eq:AP2_First_partial_eq}
and using Barrow's rule with respect to $t_j$ to integrate by parts
the derivative $D^{\gamma}_t = \frac{\partial}{\partial t_j}$
(eq.~\eqref{eq:jder}), we conclude that
\begin{equation}
  \label{eq:Sp1}
   \mbox{$S^1(\tau;\kappa, \gamma)$ is a smooth function of $\tau$ for $\tau\in (0,1)^{n-1}$}.
\end{equation}
The other two integrals in~\eqref{eq:AP2_First_partial_eq} are
suitable for differentiation under the integral sign, as shown as part
of the paragraph entitled ``inductive step'' below.

We now turn to the integral in~\eqref{eq:AP2_main_term_alp_gen}, which
we denote by $I = I(\tau)$, and we first show that, for $\varphi \in C^m(R_\varepsilon)$, the corresponding
integrand is integrable. To do this, using~\eqref{eq:bold_alphabeta},  for a
given $q$ ($1\leq q\leq m$) and a
given~$(\boldsymbol{\alpha},\boldsymbol{\beta})\in (\mathbb{N}_0^{n-1})^{2q}$, we define the quantity
\begin{equation}
  \label{eq:htilde_def}
  \widetilde{h}_{\boldsymbol{\alpha},\boldsymbol{\beta}}(\tau,t) =\widetilde{h}^q_{\boldsymbol{\alpha},\boldsymbol{\beta}}(\tau,t) =\prod_{k=1}^{q}
h_{\alpha^{(k)},\beta^{(k)}}(\tau,t);
\end{equation}
for notational simplicity the super-index $q$ in
$\widetilde{h}^q_{\boldsymbol{\alpha},\boldsymbol{\beta}}$ is
suppressed in what follows. Using the product differentiation rule to re-express the
first derivative in the integrand of $I(\tau)$ we then obtain
\begin{equation}
    I(\tau) = \int_{R_\varepsilon}  \left[ \frac{ (t_n)^s D^{\gamma}_\tau \widetilde{h}_{\boldsymbol{\alpha},\boldsymbol{\beta}}(\tau,t)}{|r(\tau,0)-r(t)|^{c_q}} - c_q\frac{ (t_n)^s (r(\tau,0)-r(t))\cdot D^{\gamma}_\tau r(\tau,0) }{|r(\tau,0)-r(t)|^{c_{(q+1)}}}\widetilde{h}_{\boldsymbol{\alpha},\boldsymbol{\beta}}(\tau,t)  \right]D^{\eta}_t \varphi(t)  dt. \label{eq:AP2_der_int11}
\end{equation}
Using once again the product-differentiation rule we see that the
derivative
$D^{\gamma}_\tau\widetilde{h}_{\boldsymbol{\alpha},\boldsymbol{\beta}}(\tau,t)$
is given by the sum of $q$ terms, each one of which equals the product
of $q$ factors, namely, $q-1$ functions of the form
$h_{\alpha^{(k)},\beta^{(k)}}(\tau,t)$ and a derivative factor of the
form $D^\gamma_\tau
h_{\alpha^{(k)},\beta^{(k)}}(\tau,t)$. But~\eqref{eq:AP2_defn_cq_h2}
and the second inequality in~\eqref{eq:rh_bounds} tell us that each
factor $h_{\alpha^{(k)},\beta^{(k)}}(\tau,t)$ (resp.
$D^\gamma_\tau h_{\alpha^{(k)},\beta^{(k)}}(\tau,t)$) is bounded by a
constant times $|(\tau,0)-t|^2$ (resp. a constant times
$|(\tau,0)-t|$). We thus obtain the relations
\begin{equation}\label{eq:htil_bounds2}
    |\widetilde{h}_{\boldsymbol{\alpha},\boldsymbol{\beta}}(\tau,t)| \le \widetilde{C}_1 |(\tau,0)-t|^{2q} \quad \text{ and } \quad  |D^{\gamma}_\tau \widetilde{h}_{\boldsymbol{\alpha},\boldsymbol{\beta}}(\tau,t)| \le \widetilde{C}_2 |(\tau,0)-t|^{2q-1},
\end{equation}
for some constants $\widetilde{C}_1$ and $\widetilde{C}_2$, which we
use in what follows to bound the two terms within square brackets
in~\eqref{eq:AP2_der_int11} by an integrable function. Indeed,
using~\eqref{eq:htil_bounds2} we obtain the estimates,
\begin{equation}
    \left|\frac{ (t_n)^s (r(\tau,0)-r(t))\cdot D^{\gamma}_\tau r(\tau,0) }{|r(\tau,0)-r(t)|^{c_{(q+1)}}}\widetilde{h}_{\boldsymbol{\alpha},\boldsymbol{\beta}}(\tau,t) \right| \le \text{const.}  \cdot \frac{|(\tau,0)-t|^s \cdot |(\tau,0)-t|^{2q+1}}{|(\tau,0)-t|^{n+2s+2q}} = \frac{\text{const.}}{|(\tau,0)-t|^{n-(1-s)}}
\end{equation}
and
\begin{equation}
    \left|\frac{ (t_n)^s D^{\gamma}_\tau\widetilde{h}_{\boldsymbol{\alpha},\boldsymbol{\beta}}(\tau,t)}{|r(\tau,0)-r(t)|^{c_q}} \right| \le \text{const.}  \cdot \frac{|(\tau,0)-t|^s \cdot |(\tau,0)-t|^{2q-1}}{|(\tau,0)-t|^{n+2s+2(q-1)}} = \frac{\text{const.}}{|(\tau,0)-t|^{n-(1-s)}},
\end{equation}
whose right-hand sides are integrable functions of $t$. Together with
the first inequality in~\eqref{eq:rh_bounds}, and since $|\eta|\leq m$
and $\varphi \in C^m(R_\varepsilon)$, these bounds tell us that the
integrand of~\eqref{eq:AP2_main_term_alp_gen} is integrable, as
claimed.

Having shown that the integrand in~\eqref{eq:AP2_main_term_alp_gen} is
integrable (which was accomplished by considering the expression
$I(\tau)$ in~\eqref{eq:AP2_der_int11} of the
integral~\eqref{eq:AP2_main_term_alp_gen}), we now obtain the desired
alternative representation of~\eqref{eq:AP2_main_term_alp_gen}, or,
equivalently, of $I(\tau)$, for
$\varphi \in C^{m+1}(R_{\varepsilon})$. To do we re-express $I(\tau)$
in the form
\begin{equation}\label{eq:AP2_partial_of_main_term_addsub}
    I(\tau) = -c_q\int_{R_\varepsilon} \frac{(t_n)^s  \widetilde{h}_{\boldsymbol{\alpha},\boldsymbol{\beta}}(\tau,t)(r(\tau,0)-r(t))\cdot D^{\gamma}_\tau r(\tau,0)}{|r(\tau,0)-r(t)|^{c_{(q+1)}}}D^{\eta}_t\varphi(t) dt+\int_{R_\varepsilon} \frac{(t_n)^s  D_{\tau}^{\gamma}\widetilde{h}_{\boldsymbol{\alpha},\boldsymbol{\beta}}(\tau,t)}{|r(\tau,0)-r(t)|^{c_{q}}}D^{\eta}_t\varphi(t)dt.
\end{equation} 
Then, adding and subtracting $D^{\gamma}_t r(t)$ from
$D^{\gamma}_\tau r(\tau,0)$ in the first integral
in~\eqref{eq:AP2_partial_of_main_term_addsub}, and adding and
subtracting
$D_{t}^{\gamma}
\widetilde{h}_{\boldsymbol{\alpha},\boldsymbol{\beta}}(\tau,t)$ from
$D_{\tau}^{\gamma}
\widetilde{h}_{\boldsymbol{\alpha},\boldsymbol{\beta}}(\tau,t)$ in the
second integral in~\eqref{eq:AP2_partial_of_main_term_addsub}, we
obtain
\begin{align}
    I(\tau) =&  -c_q \int_{R_\varepsilon} \frac{(t_n)^s  \widetilde{h}_{\boldsymbol{\alpha},\boldsymbol{\beta}}(\tau,t)h_{0,\gamma}(\tau,t)}{|r(\tau,0)-r(t)|^{c_{(q+1)}}}D^{\eta}_t\varphi(t) dt + \int_{R_\varepsilon} \frac{(t_n)^s  (D^{\gamma}_\tau \widetilde{h}_{\boldsymbol{\alpha},\boldsymbol{\beta}} (\tau,t)+D^{\gamma}_t \widetilde{h}_{\boldsymbol{\alpha},\boldsymbol{\beta}} (\tau,t))}{|r(\tau,0)-r(t)|^{c_{q}}}D^{\eta}_t\varphi(t)dt \label{eq:AP2_partial_of_main_term_combineback0}\\
    &-c_q\int_{R_\varepsilon} \frac{(t_n)^s \widetilde{h}_{\boldsymbol{\alpha},\boldsymbol{\beta}}(\tau,t)(r(\tau,0)-r(t))\cdot D^{\gamma}_t r(t)}{|r(\tau,0)-r(t)|^{c_{(q+1)}}}D^{\eta}_t\varphi(t) dt-\int_{R_\varepsilon} \frac{(t_n)^s D^{\gamma}_t \widetilde{h}_{\boldsymbol{\alpha},\boldsymbol{\beta}}(\tau,t) }{|r(\tau,0)-r(t)|^{c_{q}}} D^{\eta}_t\varphi(t)dt. \label{eq:AP2_partial_of_main_term_combineback}
\end{align}
wherein, all integrands are integrable, as it follows by employing
once again the bounds in~\eqref{eq:rh_bounds}.
Note that the subtraction of the term $D^{\gamma}_t r(t)$ from the
first integrand in~\eqref{eq:AP2_partial_of_main_term_addsub} results
in the term $h_{0,\gamma}(\tau,t)$, which weakens the kernel's
singularity and thereby permits a subsequent differentiation under the
integral sign, that is to be performed as part of the inductive step.
Furthermore, as detailed below, the replacement of the term
$D^{\gamma}_\tau
\widetilde{h}_{\boldsymbol{\alpha},\boldsymbol{\beta}}$ by the
subtracted term
$D^{\gamma}_t \widetilde{h}_{\boldsymbol{\alpha},\boldsymbol{\beta}}$
in the second integrand of~\eqref{eq:AP2_partial_of_main_term_addsub}
enables an integration-by-parts step in that integral. The
corresponding added term combines with the original term
$D^{\gamma}_\tau
\widetilde{h}_{\boldsymbol{\alpha},\boldsymbol{\beta}}$ to yield the
sum
$(D^{\gamma}_\tau
\widetilde{h}_{\boldsymbol{\alpha},\boldsymbol{\beta}}+D^{\gamma}_t
\widetilde{h}_{\boldsymbol{\alpha},\boldsymbol{\beta}})$ which
vanishes to the same order as
$\widetilde{h}_{\boldsymbol{\alpha},\boldsymbol{\beta}}$ itself. Indeed, in view
of the identity
\begin{equation}\label{eq:sum_htilde_partials}
  D^{\gamma}_\tau\widetilde{h}_{\boldsymbol{\alpha},\boldsymbol{\beta}}(\tau,t)+D^{\gamma}_t\widetilde{h}_{\boldsymbol{\alpha},\boldsymbol{\beta}}(\tau,t)  =\sum_{i=1}^{q}\left( h_{\alpha^{(i)}+\gamma,\beta^{(i)}}(\tau,t)+ h_{\alpha^{(i)},\beta^{(i)}+\gamma}(\tau,t)\right) \prod_{\substack{k=1\\ k\ne i }}^{q} h_{\alpha^{(k)},\beta^{(k)}}(\tau,t)
\end{equation}
we see that if
$\widetilde{h}_{\boldsymbol{\alpha},\boldsymbol{\beta}}$ contains a
product of functions $h_{\alpha^{(k)},\beta^{(k)}}$ with
$(\boldsymbol{\alpha},\boldsymbol{\beta}) \in I_{m,q}(\kappa)$,
$1 \le q \le m$, then the right-hand side
of~\eqref{eq:sum_htilde_partials} contains products of functions
$h_{\alpha^{(k)},\beta^{(k)}}$ with
$(\boldsymbol{\alpha},\boldsymbol{\beta}) \in
I_{m+1,q}(\kappa+\gamma)$, $1 \le q \le m+1$---which indeed results in
an additionally weakened integrand singularity.

To perform the aforementioned integration-by-parts step we  employ the
product differentiation rule and thereby re-express the two integrals
in~\eqref{eq:AP2_partial_of_main_term_combineback} in the form
\begin{align}
     I(\tau) =  &-c_q \int_{R_\varepsilon} \frac{(t_n)^s h_{0,\gamma}(\tau,t)\widetilde{h}_{\boldsymbol{\alpha},\boldsymbol{\beta}}(\tau,t)}{|r(\tau,0)-r(t)|^{c_{(q+1)}}} D^{\eta}_t\varphi(t) dt + \int_{R_\varepsilon} \frac{(t_n)^s (D_{\tau}^{\gamma} \widetilde{h}_{\boldsymbol{\alpha},\boldsymbol{\beta}}+D^{\gamma}_t \widetilde{h}_{\boldsymbol{\alpha},\boldsymbol{\beta}}) (\tau,t)}{|r(\tau,0)-r(t)|^{c_{q}}} D^{\eta}_t\varphi(t)dt \label{eq:AP2_partial_of_main_term_gen_1} \\
    &+\int_{R_\varepsilon} \frac{(t_n)^s   \widetilde{h}_{\boldsymbol{\alpha},\boldsymbol{\beta}}(\tau,t)}{|r(\tau,0)-r(t)|^{c_{q}}}D^{\eta+\gamma}_t\varphi(t) dt - \int_{R_\varepsilon} D^{\gamma}_t\left( \frac{(t_n)^s   \widetilde{h}_{\boldsymbol{\alpha},\boldsymbol{\beta}}(\tau,t)D^{\eta}_t\varphi(t)}{|r(\tau,0)-r(t)|^{c_{q}}}  \right)dt.\label{eq:AP2_partial_of_main_term_gen_2}
\end{align}
which provides the desired alternative expression for the integral
in~\eqref{eq:AP2_main_term_alp_gen}.  Using Barrow's rule, and in
direct analogy with~\eqref{eq:Sp1}, we obtain
\begin{equation}
  \label{eq:Sp2}
  S^2(\tau;\eta, \gamma, q, \boldsymbol{\alpha},\boldsymbol{\beta}) =  \int_{R_\varepsilon} D^{\gamma}_t\left( \frac{(t_n)^s  \widetilde{h}_{\boldsymbol{\alpha},\boldsymbol{\beta}}(\tau,t) D^{\eta}_t\varphi(t)}{|r(\tau,0)-r(t)|^{c_{q}}}  \right) dt\quad\mbox{is a smooth function of $\tau$ for $\tau\in (0,1)^{n-1}$}.
\end{equation}
Therefore, combining equation~\eqref{eq:AP2_First_partial_eq}--\eqref{eq:Sp1} and
\eqref{eq:AP2_partial_of_main_term_gen_1}--\eqref{eq:Sp2}, and exploiting 
the observation following equation~\eqref{eq:sum_htilde_partials}, we
can re-express~\eqref{eq:AP2_main_term_alp_gen}--\eqref{eq:AP2_smooth_term_alp_gen}
in the form
\begin{align}
     D^{\kappa+\gamma}_\tau \widetilde{F}(\tau) &= \sum_{q=1}^{m+1} \sum_{\substack{|\eta| \leq m+1-q \\ \eta \leq \kappa+\gamma}} \sum_{(\boldsymbol{\alpha},\boldsymbol{\beta})\in \mathcal{I}_{m+1,q}(\kappa+\gamma)} C^{\kappa+\gamma}_{_{q,\eta,\boldsymbol{\alpha},\boldsymbol{\beta}}}\int_{R_\varepsilon} (t_n)^s  \frac{ \widetilde{h}_{\boldsymbol{\alpha},\boldsymbol{\beta}}(\tau,t)}{|r(\tau,0)-r(t)|^{c_q}}  D^{\eta}_t\varphi(t)dt
    \label{eq:AP2_main_term_alp_gen_rexpressed1} \\
    &\quad  +\int_{R_\varepsilon} \frac{(t_n)^s}{|r(\tau,0)-r(t)|^{c_0}} D^{\kappa+\gamma}_t\varphi(t) dt+ \widetilde S(\tau; \kappa,\gamma). \label{eq:AP2_main_term_alp_gen_rexpressed2}
\end{align}
Note that, in particular, the constants
$C^{\kappa+\gamma}_{_{q,\eta,\boldsymbol{\alpha},\boldsymbol{\beta}}}$
absorb the coefficients
$C^{\kappa}_{q,\gamma,\boldsymbol{\alpha},\boldsymbol{\beta}}$
in~\eqref{eq:AP2_main_term_alp_gen}, and
\begin{equation} \label{eq:AP2_main_smooth_term_tilde}
    \widetilde S(\tau; \kappa,\gamma) =S(\tau; \kappa,\gamma)- S^1(\tau;\kappa, \gamma) -\sum_{q=1}^{m} \sum_{\substack{|\eta| \leq m-q \\ \eta \leq \kappa}} \sum_{(\boldsymbol{\alpha},\boldsymbol{\beta})\in \mathcal{I}_{m,q}(\kappa)} C^{\kappa}_{_{q,\eta,\boldsymbol{\alpha},\boldsymbol{\beta}}} S^2(\tau;\eta, \gamma, q,\boldsymbol{\alpha},\boldsymbol{\beta}) .
\end{equation}
is a smooth function of $\tau$ in the
interior of $R_\varepsilon$.

\paragraph{Inductive Step.}

Assuming the inductive statement holds for a given integer
$m = p \in \mathbb{N}_0$, we prove that it also holds for $m = p+1$.
Specifically, taking $\varphi \in C^{p+1}(R_\varepsilon)$, assume that the
equation~\eqref{eq:AP2_main_term_alp_gen}–\eqref{eq:AP2_smooth_term_alp_gen}
is valid for all multi-indices $\kappa$ and $\gamma$ satisfying
$|\kappa| = p$ and $|\gamma| = 1$. We then show that the same relations
hold for every $\gamma$ with $|\gamma| = 1$ and every
$\widetilde{\kappa} = \kappa + \delta \in \mathbb{N}_0^{n-1}$, where
\begin{equation}
  \label{eq:delta}
 \mbox{$\kappa$, $\delta\in \mathbb{N}_0^{n-1}$ satisfy $|\kappa| = p$ and $|\delta|=1$}
\end{equation}
(note that, in particular, $|\widetilde{\kappa}| = p+1$).  To this
end, since $\varphi \in C^{p+1}(R_\varepsilon)$, we may invoke the
alternative
representation~\eqref{eq:AP2_main_term_alp_gen_rexpressed1}--\eqref{eq:AP2_main_term_alp_gen_rexpressed2}
with $\gamma$ replaced by $\delta$. This representation---like
\eqref{eq:AP2_main_term_alp_gen}–\eqref{eq:AP2_smooth_term_alp_gen}
itself---is valid by the inductive hypothesis, and it enables us to
evaluate the $D^\gamma$ derivative of
$D^{\kappa+\delta}_\tau \widetilde{F}(\tau)$.

Substituting $\gamma$ by $\delta$ and $m$ by $p$ in~\eqref{eq:AP2_main_term_alp_gen_rexpressed1}--\eqref{eq:AP2_main_term_alp_gen_rexpressed2} we obtain
\begin{align}
     D^{\widetilde{\kappa}}_\tau \widetilde{F}(\tau) = D^{\kappa+\delta}_\tau \widetilde{F}(\tau) = \sum_{q=1}^{p+1} \sum_{\substack{|\eta| \leq p+1-q \\ \eta \leq \kappa+\delta}} &\sum_{(\boldsymbol{\alpha},\boldsymbol{\beta})\in \mathcal{I}_{p+1,q}(\kappa+\delta)} C^{\widetilde{\kappa}}_{_{q,\eta,\boldsymbol{\alpha},\boldsymbol{\beta}}}\int_{R_\varepsilon} (t_n)^s   \frac{ \widetilde{h}_{\boldsymbol{\alpha},\boldsymbol{\beta}}(\tau,t)}{|r(\tau,0)-r(t)|^{c_q}}D^{\eta}_t\varphi(t) dt
    \label{eq:AP2_main_term_alp_gen_rexpressed1-delta} \\
    &  +\int_{R_\varepsilon} \frac{(t_n)^s}{|r(\tau,0)-r(t)|^{c_0}} D^{\kappa+\delta}_t\varphi(t) dt+ \widetilde S(\tau; \kappa,\delta). \label{eq:AP2_main_term_alp_gen_rexpressed2-delta}
\end{align}
The integral in~\eqref{eq:AP2_main_term_alp_gen_rexpressed2-delta}
coincides with the expression that is obtained by replacing $\varphi$
by $D^{\kappa+\delta}\varphi$ in~\eqref{eq:AP2_Int2}. Hence, by the
same argument used to justify~\eqref{eq:ideap0} we see that the
integral in~\eqref{eq:AP2_main_term_alp_gen_rexpressed2-delta} is a
$C^1$ function of $\tau$ and its $D^{\delta}$ derivative can be
expressed in the form
\begin{equation}\label{eq:AP2_der_again}
   -c_0\int_{R_\varepsilon} \frac{(t_n)^s (r(\tau,0)-r(t))\cdot D^{\gamma}_\tau r(\tau,0)}{|r(\tau,0)-r(t)|^{c_1}} D^{\kappa+\delta}_t \varphi(t)dt.
 \end{equation}
 The differentiability of the integral
 in~\eqref{eq:AP2_main_term_alp_gen_rexpressed1-delta} with respect to
 $\tau$, together with the continuity of the corresponding
 $D^{\gamma}_\tau$ derivative and the validity of the relation
\begin{equation}\label{eq:AP2_der_int_change1}
    D^{\gamma}_\tau \left[\int_{R_\varepsilon}   \frac{ (t_n)^s \widetilde{h}_{\boldsymbol{\alpha},\boldsymbol{\beta}}(\tau,t)}{|r(\tau,0)-r(t)|^{c_q}}D^{\eta}_t \varphi(t)  dt\right] = \int_{R_\varepsilon}  D^{\gamma}_\tau \left[\frac{ (t_n)^s \widetilde{h}_{\boldsymbol{\alpha},\boldsymbol{\beta}}(\tau,t)}{|r(\tau,0)-r(t)|^{c_q}} \right] D^{\eta}_t\varphi(t) dt,
  \end{equation}
  in turn, follow from the integrability of the right-hand integrand
  in~\eqref{eq:AP2_der_int_change1} in view of
  Lemma~\ref{lem:Ap_Fubini} and Remark~\ref{rem:continuity}. The
  integrability of this integrand was established in the
  ``Preparation'' section above---since this integrand
  in~\eqref{eq:AP2_main_term_alp_gen} coincides with the integrand on
  the right-hand side of~\eqref{eq:AP2_der_int_change1}. Thus,
  $D^{\widetilde\kappa}_\tau \widetilde{F}\in C^1((0,1)^{n-1})$ and
  its $D^{\gamma}_\tau$ derivative is given by the expression
\begin{align}
     D^{\widetilde\kappa+\gamma}_\tau \widetilde{F}(\tau) &= \sum_{q=1}^{p+1} \sum_{\substack{|\eta| \leq p+1-q \\ \eta \leq \widetilde\kappa}} \sum_{(\boldsymbol{\alpha},\boldsymbol{\beta})\in \mathcal{I}_{p+1,q}(\widetilde\kappa)} C^{\widetilde\kappa}_{_{q,\eta,\boldsymbol{\alpha},\boldsymbol{\beta}}}\int_{R_\varepsilon} (t_n)^s D^{\gamma}_\tau \left[\frac{ \widetilde{h}_{\boldsymbol{\alpha},\boldsymbol{\beta}}(\tau,t)}{|r(\tau,0)-r(t)|^{c_q}}\right] D^{\eta}_t\varphi(t) dt\\
    &\quad  -c_0\int_{R_\varepsilon} \frac{(t_n)^s(r(\tau,0)-r(t))\cdot D^{\gamma}_\tau r(\tau,0) }{|r(\tau,0)-r(t)|^{c_1}} D^{\widetilde\kappa}_t \varphi(t) dt + S(\tau; \widetilde \kappa,\gamma),
\end{align}
wherein all integrands are integrable functions of $t$,
$C^{\widetilde\kappa}_{_{q,\eta,\boldsymbol{\alpha},\boldsymbol{\beta}}}$ are
constants, and
$ S(\tau; \widetilde \kappa,\gamma) = D^{\gamma}_\tau\widetilde S(\tau; \kappa, \delta)$ is a
$C^{\infty}$ function in $(0,1)^{n-1}$. In view
of~\eqref{eq:htilde_def} this completes the inductive step, and thus
concludes the proof of the lemma.
\end{proof}

\begin{lem}\label{lem:boundary_reg_V}
  Let $m$ denote a non-negative integer and let
  $f \in C^m(\overline{\Omega})$, where $\Omega\subset \mathbb{R}^n$
  ($n\geq 2$) is an open and bounded domain with an infinitely
  differentiable boundary. Then the restriction
  $\left . V[f]\right|_{\partial\Omega}$ of the Newtonian potential
  $V[f]$ in~\eqref{single_layer_vol} to $\partial\Omega$ satisfies
  $\left . V[f]\right|_{\partial\Omega}\in C^{m+1}(\partial \Omega)$.
\end{lem}
\begin{proof}

  For $n \geq 3$ and $s = 0$, the operator $F_s[\phi]$
  in~\eqref{eq:F_s-op} reduces to the Newtonian potential $V[f]$,
\begin{equation}
F_0[f](x) = \int_{\Omega} \frac{f(y)}{|x - y|^{n-2}} dy = \omega_n (2 - n) V[f](x), \quad x \in \overline{\Omega}.
\end{equation}
Thus, for $n\geq 3$ the present lemma follows directly from
Lemma~\ref{lem:boundary_reg_Fs}. For $n = 2$, the $s = 0$ version of
$F_s$ does not coincide with the corresponding 2D Poisson
potential. However the proof in Lemma~\ref{lem:boundary_reg_Fs}
remains valid if the the singular kernel in~\eqref{eq:F_s-op} (with
$s=0$) is substituted by the 2D Poisson kernel
$\frac{1}{2\pi}\log|x-y|$. To see this we first note that upon such a
substitution (taking into account that $c_0 = n + 2s - 2$), the terms
containing
\begin{equation}
\frac{1}{|x - y|^{n + 2s - 2}} \quad \text{and} \quad \frac{1}{|r(\tau, 0) - r(t)|^{n + 2s - 2}}
\end{equation}
in equations~\eqref{eq:AP2_Int1},~\eqref{eq:AP2_Int2},~\eqref{eq:I2_der_ker},~\eqref{eq:AP2_First_partial_eq},~\eqref{eq:Sp1},~\eqref{eq:AP2_main_term_alp_gen_rexpressed2}, and~\eqref{eq:AP2_main_smooth_term_tilde}  (eq.~\eqref{eq:AP2_defn_cq_h1})  are replaced by
\begin{equation}
\log |x - y| \quad \text{and} \quad \log |r(\tau, 0) - r(t)|,
\end{equation}
respectively, while all kernels containing other powers of remain
unchanged. With these modifications, the present lemma follows from
the argument in Lemma~\ref{lem:boundary_reg_Fs}  without
further modifications.
\end{proof}

\begin{remark}\label{rem:Gs}
  As mentioned above, the proof of Lemma~\ref{lem:boundary_reg_V}
  essentially coincides with the particular case $s=0$ of
  Lemma~\ref{lem:boundary_reg_Fs}: the two statements coincide exactly
  for $s=0$ and $n\geq 3$, while, for $n=2$, the latter result is
  established by an argument entirely analogous to that of the
  former. But, in fact, Lemma~\ref{lem:boundary_reg_V} can be viewed
  as the particular case $s=0$ of Lemma~\ref{lem:boundary_reg_Fs} even
  in the case $n=2$. To see this we replace the kernel
  in~\eqref{eq:F_s-op} by the kernel $(|x-y|^{-2s}-1)/(-2s)$ which
  results in the operator
  \begin{equation}
    \label{eq:tildeF_log}
    G_s = -(F_s-A_0)/(2s),
  \end{equation}
  where $A_0 = A_0[\phi]$ is a linear integral functional: $A_0[\phi]$
  is independent of $x$. It is easy to check that
  \begin{equation}
    \label{eq:log-limit}
    \lim_{s\to 0^+}G_s[\phi] = 2\pi V[\phi], 
  \end{equation}
  and that the limit is uniform over $\overline\Omega$. In view
  of~\cite[Thm.~7.17]{rudin1976}, the smoothness of the restriction of
  $V[\phi]$ to $\partial \Omega$ follows from the uniform convergence
  of the tangential derivatives of the boundary values of $G_s[\phi]$,
  which can be established by arguments analogous to those employed in
  the ``Preparation for the Inductive Step'' portion of the proof of
  Lemma~\ref{lem:boundary_reg_Fs}. Thus, the case $s=0$ with $n=2$
  also is, in essence, the particular case $s=0$ of
  Lemma~\ref{lem:boundary_reg_Fs}.
\end{remark}

\begin{remark}\label{rem:Vf_reg}
  Lemma~\ref{lem:boundary_reg_V} establishes the boundary-regularity
  result $V[f] \in C^{m+1}(\partial\Omega)$ for the Poisson potential
  $V[f]$. In fact, full-domain regularity
  $V[f] \in C^{m+1}(\overline\Omega)$ can also be obtained for
  $f\in C^m(\overline\Omega)$, using an argument similar to the one
  employed in Lemma~\ref{lem:boundary_reg_Fs} but in the case $s=0$
  for which the term $d^s(y) = 1$ is not singular.  It is relevant to
  contrast such a result with others previously available in the
  literature.  For example~\cite[Thm.~2.28]{folland2020introduction}
  establishes that, for $0<\alpha< 1$, we have
  $V[f]\in C^{k+2+\alpha}(\Omega)$ provided
  $f\in C^{k+\alpha}(\Omega)$. However, this is an interior-regularity
  result: it does not establish regularity up to the boundary. In
  fact, we are not aware of any existing results concerning regularity
  up to or on the boundary, either for potentials related to the
  Riesz-like operator~\eqref{eq:F_s-op} or for the Poisson potential
  $V[f]$.
\end{remark}


We now proceed to the proof of the main results of this paper, Theorems~\ref{thm:FL_into_Lap} through~\ref{thm:new_form_1D}.

\begin{center}{\bf Proof of Theorem~\ref{thm:FL_into_Lap}}
\end{center}

\noindent
For $x \in \mathbb{R}^n \setminus \Omega$, the result follows by
applying integration by parts to the integral over the Lipschitz
domain $\Omega$~\cite[Thm.~3.34]{mclean2000} on the right-hand side of
\eqref{eq:FL_op_grad}, and then interchanging differentiation and
integration so that the gradient acts outside the integral. We
therefore focus on the more delicate case $x \in \Omega$. Using the
dominated convergence theorem, integration by parts and
Lemma~\ref{lem:und_int}, the integral on the right side of
\eqref{eq:FL_op_grad} can be manipulated as follows:
\begin{align}
&\int_{\Omega} |x-y|^{-n-2s+2}\nabla u(y)dy = \lim_{\varepsilon \to 0}\int_{\Omega \setminus B_{\varepsilon}(x)}|x-y|^{-n-2s+2}\nabla u(y)dy \\
&=  \lim_{\varepsilon \to 0}\left[\int_{\Omega \setminus B_{\varepsilon}(x)}u(y)\nabla_{x}|x-y|^{-n-2s+2}dy  + \int_{\partial B_{\varepsilon}(x)}u(y)|x-y|^{-n-2s+2}\frac{x-y}{\varepsilon}dS_y\right]\\
  &=  P.V\int_{\Omega}u(y)\nabla_{x}|x-y|^{-n-2s+2}dy  + \lim_{\varepsilon \to 0}\int_{\partial B_{\varepsilon}(x)}u(y)|x-y|^{-n-2s+2}\frac{x-y}{\varepsilon}dS_y\\
&=\nabla_{x}  \int_{\Omega}u(y)|x-y|^{-n-2s+2}dy  - \lim_{\varepsilon \to 0}\int_{\partial B_{\varepsilon}(0)}u(x+y)\frac{y}{\varepsilon^{n+2s-1}}dS_y. 
\label{eq:grad_kernel}
\end{align} 
Using equations~\eqref{spher_1} and~\eqref{spher_elems} the last
integral in~\eqref{eq:grad_kernel} becomes
  \begin{align}
    &\int_{\partial B_{\varepsilon}(0)}u(x+y)\frac{y}{\varepsilon^{n+2s-1}}dS_y =  \frac{1}{\varepsilon^{
      2s-1}}\int_{[0,\pi]^{n-2} \times [0,2\pi]}  u(x+\varepsilon \xi(\phi,\theta))\xi(\phi,\theta)F(\phi)d\phi d\theta \label{eq:int_taylor_exp_u0} \\
    &= \frac{1}{\varepsilon^{
      2s-1}}\left(\!\int_{[0,\pi]^{n-2} \times [0,\pi]} +\int_{ [0,\pi]^{n-2} \times [\pi,2\pi]}\!\right) u(x+\varepsilon \xi(\phi,\theta))\xi(\phi,\theta)F(\phi)d\phi d\theta . \label{eq:int_taylor_exp_u}
\end{align}
Introducing the change of variables $\theta \mapsto \theta+\pi$ and
$\phi_j \mapsto \pi-\phi_j$ for each $j=1,\cdots,n-2$ in the second
integral in~\eqref{eq:int_taylor_exp_u} and using Taylor's theorem
$u(x\pm h) = u(x) \pm \nabla u(\mu^\pm) \cdot h$ (where $\mu^\pm_i$
lies between $x_i$ and $x_i\pm h_i$), the
quantity~\eqref{eq:int_taylor_exp_u} may be re-expressed in the form
\begin{align}
     &\varepsilon^{1-2s}\int_{[0,\pi]^{n-2} \times [0,\pi]} (u(x+\varepsilon \xi(\phi,\theta)) - u(x-\varepsilon \xi(\phi,\theta)))\xi(\phi,\theta) F(\phi)d\phi d\theta  \\
     &=\varepsilon^{2(1-s)} \int_{ [0,\pi]^{n-2} \times [0,\pi]} \! \! \left( \nabla u(\mu^+)\cdot \xi(\phi,\theta)+\nabla u(\mu^-)\cdot \xi(\phi,\theta)\right)\xi(\phi,\theta) F(\phi) d\phi d\theta,  \label{eq:bd_int_z}
\end{align}
which tells us that the second term in~\eqref{eq:grad_kernel}
vanishes. Thus,
\begin{equation}\label{eq:grad_u_out}
  \int_{\Omega} |x-y|^{-n-2s+2}\nabla u(y)dy = \nabla_{x}  \int_{\Omega}u(y)|x-y|^{-n-2s+2}dy.
\end{equation}
Using Lemma~\ref{lem:FL_op_grad} and taking divergence of~\eqref{eq:grad_u_out} we obtain
\begin{equation}
(-\Delta)^{s} u(x)=C_{n,s} \nabla_{x} \cdot \nabla_{x}  \int_{\Omega}u(y)|x-y|^{-n-2s+2}dy = C_{n,s}\Delta_{x}\int_{\Omega} \frac{u(y)}{|x-y|^{n+2s-2}} dy,
\end{equation}
as desired. The proof is complete.\hfill\qedsymbol

\begin{center}{\bf Proof of Theorem~\ref{thm:new_form_nD}}
\end{center}

\noindent
To show that equations~\eqref{eq:reformul1}--\eqref{eq:reformul2}
admit a solution satisfying~\eqref{eq:uniq_triple} we consider the
solution $u$ of~\eqref{frac_lapl} corresponding to the given function
$f \in C^{m+0}(\overline{\Omega})$, $m \geq 2$. To employ the
relation~\eqref{eq:frac-from-clas} we first verify that $u$ satisfies
the hypotheses of Theorem~\ref{thm:FL_into_Lap}. Indeed, by
Theorem~\ref{thm:fu_rel}, the regularity condition
$f \in C^{2+0}(\overline{\Omega})$ ensures that $u$ admits the
representation $u(x) = d^s(x)\phi(x)$, as in
equation~\eqref{eq:FL-sol-ws}, with
$\phi \in C^{2}(\overline{\Omega})$, and, thus, in particular,
$u \in C^{2}(\Omega) \cap C(\mathbb{R}^n)$.  Noting that $d^{s-1}$ is
an integrable function in $\overline{\Omega}$, it follows that each
first-order partial derivative of $u$ is integrable, and, thus, $u$
satisfies all the hypotheses required by
Theorem~\ref{thm:FL_into_Lap}, as desired.

We may thus use equation~\eqref{eq:frac-from-clas}, which, in view
of~\eqref{eq:F_s-op}, tells us that $v =C_{n,s} F_{s}[\phi]$ is a
solution of the problem~\eqref{eq:poisson} with $g_1 = f$ and
$g_2 = C_{n,s}F_{s}[\phi]\big |_{\partial\Omega}$. It follows that
$\widetilde v = v -V[f]$ (eq.~\eqref{eq:v-poten}) solves the Laplace
problem~\eqref{eq:laplace} with
$h= C_{n,s}F_{s}[\phi]\big |_{\partial\Omega} -V[f]\big
|_{\partial\Omega}$.  But, since
$f\in C^{m+0}(\overline{\Omega})\subset C^m(\overline{\Omega})$, in
view of Lemma~\ref{lem:boundary_reg_V} we have
$\left . V[f]\right|_{\partial\Omega}\in C^{m+1}(\partial
\Omega)$. Further, since the function $\phi$ provided by
Theorem~\ref{thm:fu_rel} satisfies
$\phi \in C^{m+s-0}(\overline{\Omega}) \cap C^{m+2s-0}(\Omega)$, and
therefore $\phi \in C^m(\overline{\Omega})$,
Lemma~\ref{lem:boundary_reg_Fs} is applicable and it tells us that
$F_s[\phi] \in C^{m+1}(\partial \Omega)$---and, therefore
$h \in C^{m+1}(\partial \Omega)$.  We have thus shown that
$\widetilde v$ satisfies~\eqref{eq:laplace} with
$h \in C^{m+1}(\partial \Omega)$. Consequently, using the solutions
$\beta_j$ ($1 \le j \le n_h$) of~\eqref{eq:betaj_def}, per
Section~\ref{prelim} we see that $\widetilde v$ admits the
representation~\eqref{eq:LP_to_Inteq}, where
$(\zeta, a) \in C(\partial \Omega) \times \mathbb{C}^{n_h}$ is a
solution of~\eqref{eq:inteq_lap}. In other words, as claimed,
$(\phi, \zeta, a)$ satisfies
equations~\eqref{eq:reformul1}--\eqref{eq:reformul2}.

To establish~\eqref{eq:uniq_triple} it remains for us to show
that $\zeta\in C^{m+1-0}(\partial \Omega)$. To do this we first
re-express equation~\eqref{eq:inteq_lap} in the form
\begin{equation}\label{eq:DLP_inverse}
\frac{1}{2}\zeta(x) =  h(x) - D[\zeta](x) - \sum_{j=1}^{n_h} a_j S[\beta_j](x), \quad x \in \partial\Omega,
\end{equation}
and we note that, as part of the proof of~\cite[Theorem
3.5]{kress1999} it is established that for any $0 < \delta < 1$, the
operator $D$ maps $C(\partial\Omega)$ to
$C^{0,\delta}(\partial\Omega)$ (continuously). In particular, in view
of~\eqref{eq:betaj_def}, it follows that the right-hand side
in~\eqref{eq:DLP_inverse} is an element of
$C^{0,\delta}(\partial\Omega)$ and therefore,
$\zeta \in C^{0,\delta}(\partial\Omega)$. Furthermore, since $D$ maps
$C^{\ell,\delta}(\partial\Omega)$ to
$C^{\ell+1,\delta}(\partial\Omega)$ (continuously) for all integers
$\ell \geq 0$~\cite[Theorem 7.1]{cristoforis2024}, and since
$h \in C^{m+1}(\partial \Omega)$, using~\eqref{eq:DLP_inverse} in a
bootstrapping argument shows that
$\zeta \in C^{m,\delta}(\partial\Omega)$ for all $0 < \delta < 1$. It
follows that equations~\eqref{eq:reformul1}–\eqref{eq:reformul2} admit
a solution $(\phi, \zeta, a)$ with
$\phi \in C^{m+s-0}(\overline{\Omega}) \cap C^{m+2s-0}(\Omega)$ and
$\zeta\in C^{m+1-0}(\partial \Omega)$, as claimed.

Regarding the uniqueness claims in the theorem, finally, let
$(\phi, \zeta, a) \in C^{2}(\overline{\Omega}) \times
C(\partial\Omega)\times \mathbb{C}^{n_h}$ denote a solution
of~\eqref{eq:reformul1}-\eqref{eq:reformul2}, and let $u(x)$ be given
by~\eqref{eq:FL-sol-ws}. Clearly, $u$ satisfies the hypotheses of
Theorem~\ref{thm:FL_into_Lap}, and thus~\eqref{eq:frac-from-clas}
holds. In view of~\eqref{eq:poisson_solve},~\eqref{eq:frac-from-clas}
and~\eqref{eq:F_s-op}, and noting that the single- and double layer
potentials in~\eqref{eq:reformul1} are harmonic functions at all
$x\in\Omega$ (since so is $N(x,y)$ for $x\ne y$), applying the
operator $\Delta_x$ to both sides of~\eqref{eq:reformul1} yields
$(-\Delta)^s u(x) = f(x)$ for $x \in \Omega$. Since, by
Theorem~\ref{thm:fu_rel}, the solution $u$ is uniquely determined, it
follows that so is the $\phi$ component of the solution vector
$(\phi, \zeta, a)$. Furthermore, as discussed
following~\eqref{eq:inteq_lap}, the constant vector $a$ is uniquely
determined, and $\zeta$ is uniquely determined if and only if
$n_h = 0$. The proof is therefore complete.  \hfill\qedsymbol
\newpage

\begin{center}{\bf Proof of Theorem~\ref{thm:new_form_1D}}
\end{center}
The proof Theorem~\ref{thm:new_form_1D}, which parallels the one for
Theorem~\ref{thm:new_form_nD}, is based on an $n=1$ version of
Theorem~\ref{thm:FL_into_Lap}, namely,
\begin{empheq}[left={\empheqlbrace}]{alignat=2}
  &(-\Delta)^su(x) = \frac{d^2}{dx^2}\int_{a}^{b}|x-y|^{1-2s}u(y)dy, \quad s \ne \frac{1}{2}  \\
  &(-\Delta)^su(x) =
  \frac{d^2}{dx^2}\int_{a}^{b}\log{|x-y|}u(y)dy,\quad s = \frac{1}{2}.
\end{empheq}
The proofs of these relations proceed in the same manner as the proof
of Theorem~\ref{thm:FL_into_Lap}, making use of one-dimensional
counterparts, established in~\cite[Lemmas~2.3--2.4]{oscar2018}, of
Lemmas~\ref{lem:FL_op_grad} and~\ref{lem:und_int}.  The remainder of
the argument then follows exactly as in the proof of
Theorem~\ref{thm:new_form_nD}. \hfill\qedsymbol
\section{Numerical illustrations\label{sec:num-ex}}

An effective numerical algorithm for the solution of the FL
boundary-value problem~\eqref{frac_lapl} has been devised following
the lines of Theorem~\ref{thm:new_form_nD} and the overall theoretical
framework presented in this paper. A key feature of the algorithm lies
in its inversion of the Laplace operator via layer potentials, which
eliminates major sources of implementation complexity, numerical
instability (cancellation errors) and ill-conditioning that often
arise in the numerical solution of the FL problem. Previous numerical
approaches—including adaptive finite element methods (FEM), radial
basis functions, isogeometric methods, and Monte Carlo
simulations—have generally exhibited low convergence rates and limited
accuracy, typically yielding errors on the order of
$10^{-2}$–$10^{-3}$~\cite{lischkle2020,xu2020,burkardt2021,minden2020,acosta2017short},
except in cases involving artificially smooth test solutions. A
detailed description of the new algorithm will be presented
elsewhere. To illustrate the practical impact of the theoretical
results presented in this paper, however, this section presents
representative numerical results produced by a Julia computational
implementation of the new numerical method for both simply connected
and multiply connected domains, and including comparisons with an
exact solution and with results for a multiply connected domain
produced by the FEM algorithm~\cite{acosta2017short}.

Throughout this section $N$ denotes the total number of unknowns
employed, including volumetric unknowns over the 2D FL domain $\Omega$
and boundary unknowns over $\partial \Omega$. The solution accuracies
are quantified using the maximum relative error and the root-mean-square (RMS) error, defined by
\begin{equation}
    \varepsilon_{N,\infty} = \frac{\max_{1 \le j \le N} |u_j - u_j^{\text{ref}}|}{\max_{1 \le j \le N} |u_j^{\text{ref}}|}, 
    \quad 
    \varepsilon_{N,\mathrm{rms}} = \sqrt{\frac{1}{N}\sum_{j=1}^N |u_j - u_j^{\text{ref}}|^2},
\end{equation}
respectively, where $u^{\text{ref}}$ is taken to equal the exact
solution, when available, or, otherwise, a reference solution computed
numerically on a fine grid. The numerical order of convergence (noc) for a given
discretization size $N=N_2$ in the various tables is calculated as
\begin{equation}
  \mathrm{noc} = \frac{\log\!\left(\varepsilon_{N_1,\infty}/ \varepsilon_{N_2,\infty}\right)}{\log\left(\sqrt{N_2/N_1}\right)},
\end{equation}
where $N=N_1$ is the number of unknowns in the experiment tabulated in
the row immediately preceding the row for $N=N_2$.

\begin{figure}[H]
\centering
\begin{minipage}{0.5\textwidth}
  \centering
  \includegraphics[width=\linewidth]{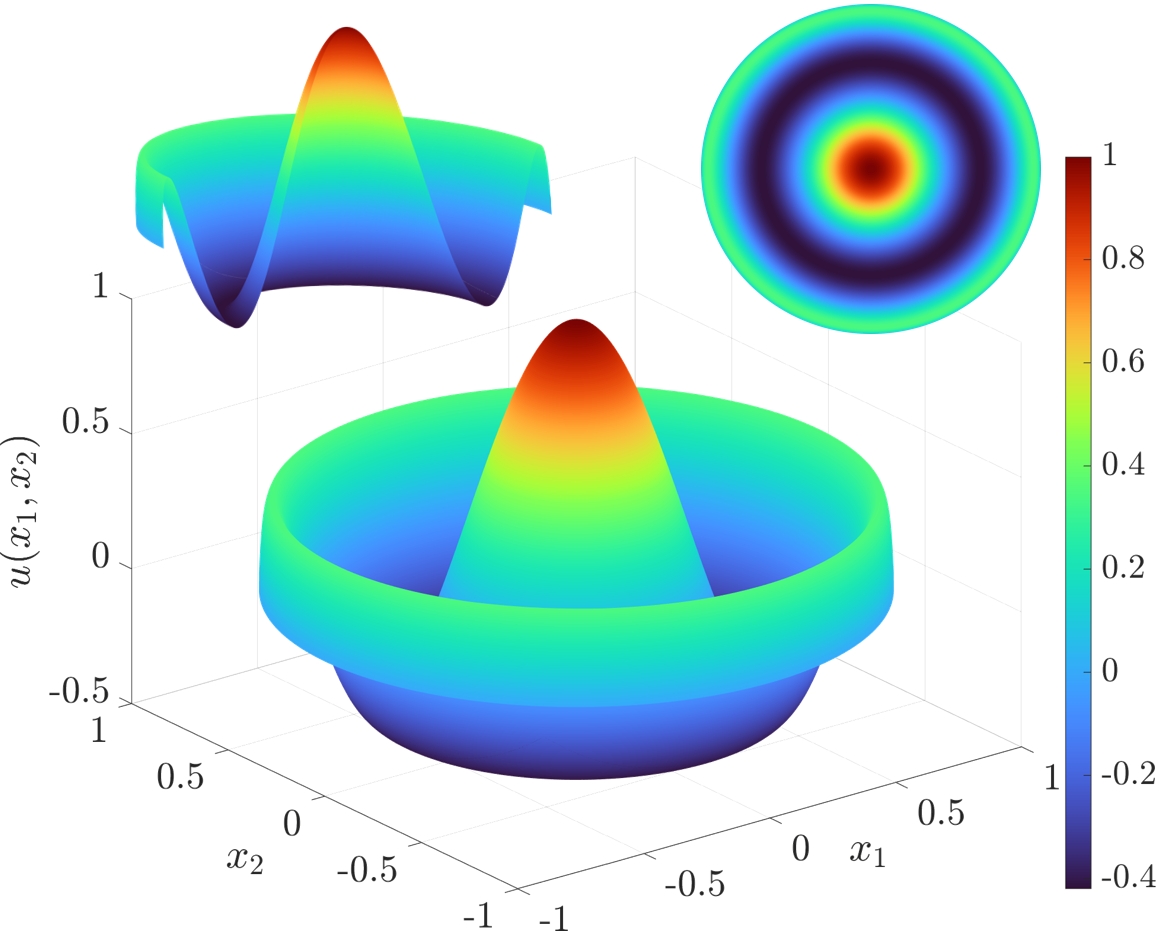} 
\end{minipage}
\hspace{2mm}
\begin{minipage}{0.4\textwidth}
    \begin{table}[H]
        \centering
        \renewcommand{\arraystretch}{1.25}
        \scalebox{0.8}{\begin{tabular}{|c|c|c|c|c|}\hline
           $N$  &  $\varepsilon_{N,\infty}$ & $\varepsilon_{N,\mathrm{rms}}$ & noc & time (sec) \\\hline
          $96$   & $2.70 \times 10^{-1}$  & $7.57 \times 10^{-2}$  & \text{--} & 0.290 \\
        $145$  & $2.00 \times 10^{-2}$  & $5.74 \times 10^{-3}$  & 12.6      & 0.457 \\
        $204$  & $7.47 \times 10^{-4}$  & $2.56 \times 10^{-4}$  & 19.2      & 0.677 \\
        $273$  & $2.59 \times 10^{-5}$  & $7.56 \times 10^{-6}$  & 23.0      & 1.278 \\
        $441$  & $1.05 \times 10^{-6}$  & $2.12 \times 10^{-7}$  & 13.4      & 2.261 \\
        $768$  & $5.66 \times 10^{-9}$  & $1.35 \times 10^{-9}$  & 18.8      & 9.832 \\
        $2080$ & $6.22 \times 10^{-10}$ & $7.17 \times 10^{-11}$ & 4.43      & 35.58 \\
        $2976$ & $5.15 \times 10^{-11}$ & $5.36 \times 10^{-12}$ & 13.9      & 59.32 \\
        $3720$ & $1.86 \times 10^{-11}$ & $2.00 \times 10^{-12}$ & 9.1       & 85.58 \\
        $4464$ & $8.46 \times 10^{-12}$ & $6.88 \times 10^{-13}$ & 8.7       & 114.0 \\
        $5328$ & $5.14 \times 10^{-13}$ & $4.55 \times 10^{-14}$ & 31.5      & 266.7 \\\hline
        \end{tabular}}
        \captionof{table}{\label{tab:disc_solution}}
    \end{table}
\end{minipage}
\caption{Left: side, top, and slice views of the solution $u$ of the
  FL problem~\eqref{frac_lapl} on the domain $\Omega$ equal to the
  unit disc, with right-hand function $f$ as described in the
  text. Right: Numerical errors, noc and computing times resulting
  from a Julia implementation of the proposed solver (which is based
  on the theoretical framework presented in this
  paper).\label{fig:disc_solution}}
\end{figure}

\begin{figure}[H]
\centering
\begin{minipage}{0.59\textwidth}
  \centering
  \includegraphics[width=\linewidth]{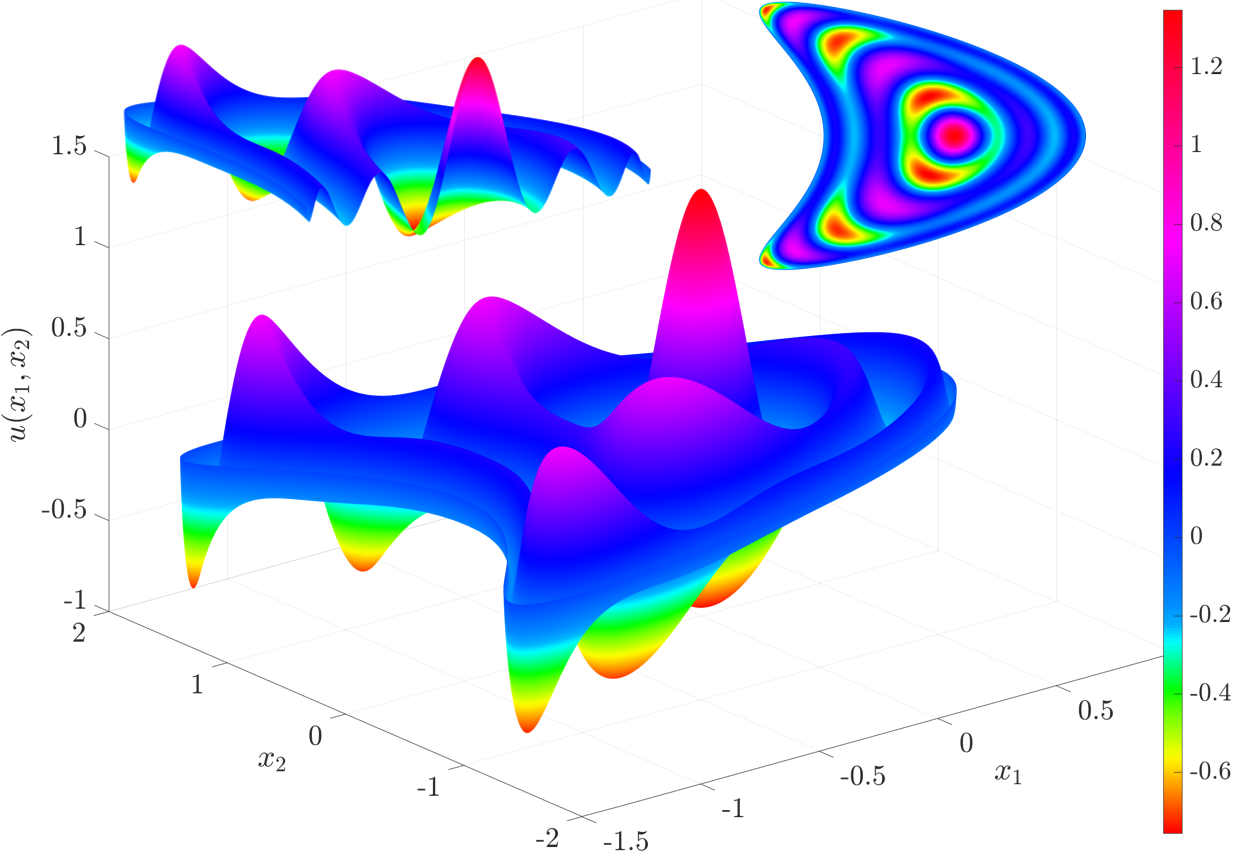} 
\end{minipage}
\begin{minipage}{0.4\textwidth}
  \begin{table}[H]
    \centering
    \renewcommand{\arraystretch}{1.25}
    \scalebox{0.8}{\begin{tabular}{|c|c|c|c|} \hline 
    $N$  &  $\varepsilon_{N,\infty}$ & $\varepsilon_{N,\mathrm{rms}}$ & noc \\ \hline 
    $1300$  & $6.48 \times 10^{-2}$  & $1.12 \times 10^{-2}$  & --   \\ 
    $1848$  & $8.91 \times 10^{-4}$  & $1.49 \times 10^{-4}$  & 24.4 \\
    $3336$  & $1.00 \times 10^{-5}$  & $8.05 \times 10^{-7}$  & 15.2 \\ 
    $9528$  & $4.89 \times 10^{-7}$  & $7.24 \times 10^{-8}$  & 5.8  \\ 
    $11568$ & $6.10 \times 10^{-9}$  & $9.35 \times 10^{-10}$ & 45.2 \\ 
    $23544$ & $3.61 \times 10^{-10}$ & $3.03 \times 10^{-11}$ & 8.0  \\ 
    $36336$ & $3.67 \times 10^{-11}$ & $1.15 \times 10^{-12}$ & 10.5 \\ 
    $38352$ & $1.29 \times 10^{-11}$ & $4.70 \times 10^{-13}$ & 38.7 \\ \hline \end{tabular}}
    \captionof{table}{\label{tab:kite_solution}}
\end{table}
\end{minipage}
\caption{Left: Same as the left panel of
  Figure~\ref{fig:disc_solution} but for the kite-shaped domain
  $\Omega$ described in the text. Right: Numerical errors and noc
  resulting from the proposed solver.\label{fig:kite_solution} }
\end{figure}

Three 2D numerical examples for problem~\eqref{frac_lapl} are
presented in
Figures~\ref{fig:disc_solution}--\ref{fig:annulus_solution}. For each
example, we report numerical errors under various discretizations and
provide visualizations of the solution $u$, including side, top, and
slice views. In all cases, the right-hand sides $f$ are chosen as
infinitely differentiable functions of the form
\begin{equation}\label{eq:f_form_num}
f(x_1,x_2) = 2^{2s} \frac{\Gamma(s+k+1)^2}{(k!)^2} (-1)^k
P_k^{(s,0)}\bigl(2g(x_1,x_2)-1\bigr),
\end{equation}
for various choices of $g=g(x_1,x_2)$, $s$, and $k$, where
$P_k^{(s,0)}(x)$ denotes the Jacobi polynomial of degree
$k$~\cite{szeg1939orthogonal}. We emphasize that, for such
smooth right-hand side functions $f$, the  solutions $u$ retain the full
boundary singularity factor $\left(d(x_1,x_2)\right)^s$.

The first test case concerns the solution of~\eqref{frac_lapl} with
$s=0.5$ and $\Omega$ equal to the unit disc centered at the origin,
and with the right-hand function given by~\eqref{eq:f_form_num} with
$g(x_1,x_2) = x_1^2 + x_2^2$ and $k=2$. The exact solution is given by
\[
u(x_1,x_2) = u^{\mathrm{ref}}(x_1,x_2) = (-1)^k \, (1-x_1^2-x_2^2)^s \, P_k^{(s,0)}\bigl(2g(x_1,x_2)-1\bigr).
\]
The corresponding numerical results are presented in
Figure~\ref{fig:disc_solution} and
Table~\ref{tab:disc_solution}. Clearly, the error decreases rapidly
with increasing $N$, initially displaying very high convergence rates
(e.g., $\text{noc} \approx 12.6$ between $N=96$ and $N=145$), before
gradually stabilizing as the approximation reaches machine precision.
It is worth noting that, for example, using a total of $441$ unknowns
and in just $2.26$ seconds of computing time the algorithm achieves
$7$ digits of accuracy. The computational cost can be greatly reduced
by using kernel-independent acceleration techniques such as the IFGF
method~\cite{bauinger2021}, which reduces the computing cost from
$\mathcal{O}(N^2)$ operations to $\mathcal{O}(N\log N)$ operations,
but such methodologies were not incorporated in the computer code used
for the examples presented in this section.

\begin{figure}[H]
\centering
\begin{minipage}{0.55\textwidth}
  \centering
  \includegraphics[width=\linewidth]{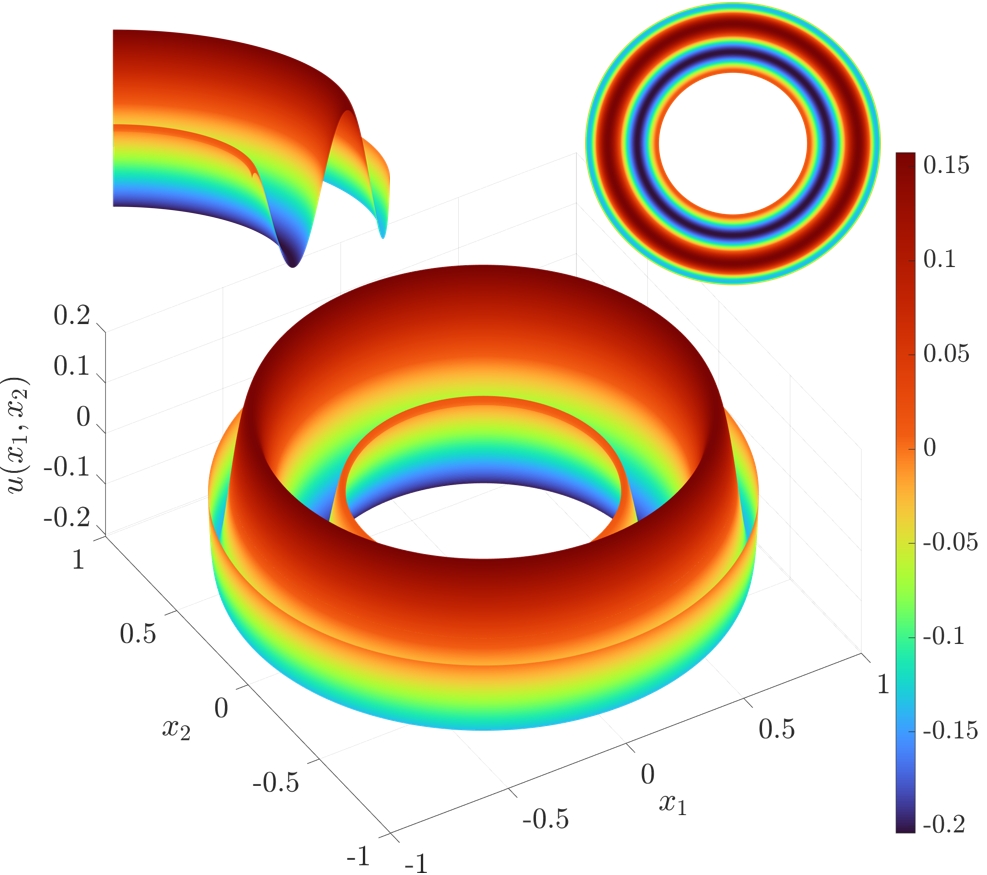} 
\end{minipage}
\hfill
\begin{minipage}{0.4\textwidth}
    \begin{table}[H]
    \centering
        \renewcommand{\arraystretch}{1.25}
        \scalebox{0.8}{\begin{tabular}{|c|c|c|c|} \hline
        $N$     & $\varepsilon_{N,\infty}$  & $\varepsilon_{N,\mathrm{rms}}$  & noc \\ \hline
        $320$   & $4.00 \times 10^{-1} $    & $3.52 \times 10^{-2} $          & -   \\ 
        $896$   & $1.03 \times 10^{-4} $    & $7.79 \times 10^{-6} $          & $16.1$  \\ 
        $1760$  & $1.81 \times 10^{-8} $    & $1.22 \times 10^{-9} $          & $25.6$  \\ 
        $2496$  & $3.73 \times 10^{-10}$    & $2.30 \times 10^{-11}$          & $22.2$  \\ 
        $6720$  & $3.32 \times 10^{-11}$    & $1.93 \times 10^{-12}$          & $4.9$   \\ 
        $9600$  & $1.61 \times 10^{-12}$    & $7.23 \times 10^{-14}$          & $17.0$  \\ \hline
        \end{tabular}}
        \captionof{table}{\label{tab:annulus_solution_RP}}
    \end{table}

  \begin{table}[H]
    \centering
        \renewcommand{\arraystretch}{1.25}
        \scalebox{0.8}{\begin{tabular}{|c|c|c|c|}\hline
            DOF &  $\#$Triangles&  Max error& noc \\ \hline
            $447$& $817$& $1.16\times 10^{-1}$& -\\
            $1800$&  $3448$&  $7.59\times 10^{-2}$& $0.305$\\
            $7217$& $14131$& $4.66\times 10^{-2}$& $0.351$\\\hline
        \end{tabular}}
        \captionof{table}{\label{tab:annulus_solution_fem}}
    \end{table}
\end{minipage}
\caption{Left: Same as the left panel of
  Figure~\ref{fig:disc_solution} but for the multiply-connected
  annular domain $\Omega$ described in the text. Upper table:
  Numerical errors and noc resulting from the proposed solver.  Lower
  table: Convergence exhibited by the FEM solution mentioned in the
  text to the reference solution produced by the proposed solver,
  providing, in particular, independent validation for the new
  formulation for multiply connected domains.  The observed
  convergence trend is consistent with the error estimate, of the
  order of $(\mbox{DOF})^{-1/2}$, reported
  in~\cite[Fig. 2]{acosta2017short} for the case $s=0.75$ considered
  here.\label{fig:annulus_solution}}
\end{figure}

For the second example we take $s=0.75$ and the right-hand
function~\eqref{eq:f_form_num} with $k=5$ and
$g(x_1,x_2) = \left( x_1 + 1.3\left( x_2/1.5 \right)^2 \right)^2 +
\left( x_2/1.5 \right)^2$, over the kite-shaped domain
$\Omega = \{ (x_1,x_2) \in \mathbb{R}^2 : g(x_1,x_2) \le 1 \}$.  The
corresponding numerical solution $u$ of~\eqref{frac_lapl} and
associated errors are displayed in Fig.~\ref{fig:kite_solution} and
Table~\ref{tab:kite_solution}, once again demonstrating fast
convergence and high accuracy. The computing times, which are larger
than those displayed in Table~\ref{tab:disc_solution} on account of
the higher oscillatory character of the solution, are consistent with
the aforementioned $\mathcal{O}(N^2)$ cost estimate, and could be
reduced by using acceleration methods such as~\cite{bauinger2021}. In
any case, we note, for reference, that the four-digit, $N=1848$
solution in Table~\ref{fig:kite_solution} was produced in a computing
time of 24.4 secs.

In our third and final example we take  $s=0.75$ utilize the multiply-connected annular domain
$$\Omega = \left\{ (x_1,x_2) \in \mathbb{R}^2 : \tfrac{1}{4} \le
  x_1^2+x_2^2 \le 1 \right\},$$ and the right-hand function $f$ given
by~\ref{eq:f_form_num} with $k=3$ and
$g(x_1,x_2) = \tfrac{1}{3}(4(x_1^2+x_2^2)-1)$. The corresponding
numerical solution $u$ of~\eqref{frac_lapl} and associated errors are
presented in the left panel of Fig.~\ref{fig:annulus_solution} and
Table~\ref{tab:annulus_solution_RP}, respectively. The errors in the
solution produced by the FEM algorithm~\cite{acosta2017short},
computed relative to reference values obtained by evaluating, on the
FEM grid, the Chebyshev approximant produced by the new numerical
method, are presented in Table~\ref{tab:annulus_solution_fem}. (It may
be useful to note that the reference solution used here coincides with
the one employed to compute the errors reported in
Table~\ref{tab:annulus_solution_RP}.) The two-digit accuracy and slow
but clear convergence exhibited by the FEM solution in
Table~\ref{tab:annulus_solution_fem}, which is additionally consistent
with the error estimate of the order of $(\mbox{DOF})^{-1/2}$ reported
in~\cite[Fig. 2]{acosta2017short} for the case $s=0.75$ considered in
this example, provide independent validation of the proposed
mathematical formalism for multiply connected domains.
\section*{Acknowledgments}
The authors gratefully acknowledge support from the Air Force Office
of Scientific Research and the National Science Foundation under
contracts FA9550-21-1-0373, FA9550-25-1-0015 and DMS-2109831.

    
\appendix
\section{Appendix: Differentiation under the integral sign\label{Ap_Fubini}}

\begin{lem}\label{lem:Ap_Fubini}
  Let $n\in\mathbb{N}$, and let $ V \subset \mathbb{R}^n$ denote a
  Lebesgue measurable set. Further let $a,b \in \mathbb{R}$ with
  $a<b$, and let $g:[a,b] \times V \to \mathbb{R}$, $g=g(\tau,t)$,
  denote a function which, for $t$ outside a set $E\subset V$ of
  measure zero, is absolutely continuous with respect to $\tau$ in the
  interval $[a,b]$.  Assume, further, that (i)~The functions $g$ and
  $\frac{\partial g}{\partial \tau}$ are integrable over $V$ for each
  fixed $\tau\in [a,b]$; that, (ii)~The integral over $V$ of
  $\frac{\partial g}{\partial \tau}$ is a continuous function of
  $\tau$ for $\tau\in [a,b]$; and that, (iii)
  \begin{equation}\label{eq:der_g_int}
   \int_a^b d\tau \int_{V} \left|\frac{\partial g}{\partial \tau}(\tau,t) \right|  dt < \infty. 
 \end{equation}
 Then, the function
 \begin{equation}
   G(\tau) = \int_{V} g(\tau,t) dt 
 \end{equation}
 is differentiable with respect to $\tau$ and
 \begin{equation}\label{eq:der_G}
   \frac{d G}{d \tau}(\tau) = \int_V \frac{\partial g}{\partial \tau}(\tau,t)dt
   \end{equation}
 \end{lem}
 \begin{proof}
   In view of~\eqref{eq:der_g_int}, we may use Fubini's theorem to write
   \begin{equation}\label{eq:fubini_applied}
     \int_a^{\tau}d\tau' \int_V \frac{\partial g}{\partial \tau}(\tau',t) dt = \int_V dt \int_a^{\tau} \frac{\partial g}{\partial \tau}(\tau',t) d\tau'
   \end{equation}
   Since $g$ is absolutely continuous with respect to $\tau$ for
   $t\not\in E$, it follows that~\cite[Ch. 5 Cor. 15]{royden1988}, for
   such values of $t$ we have,
   $ \int_a^{\tau} \frac{\partial g}{\partial \tau}(\tau',t) d\tau' =
   g(\tau,t)-g(a,t)$. Substituting this identity
   into~\eqref{eq:fubini_applied}
   yields\begin{equation}\label{eq:G_difference}
     \int_a^{\tau}d\tau'\int_V \frac{\partial g}{\partial
       \tau}(\tau',t) \,dt=G(\tau)-G(a).
   \end{equation}
    By assumption (ii), the function $\int_V \partial_{\tau}
g(\tau,t)\,dt$ in $\tau$ is continuous on $[a,b]$. Hence the left-hand side of~\eqref{eq:G_difference} is differentiable with respect to $\tau$, and differentiation yields the desired relation~\eqref{eq:der_G}.
\end{proof}

\begin{remark}\label{rem:continuity}
  The continuity assumption in point~(ii) of the previous lemma is
  readily verified for the integrals involving weakly singular kernels
  considered in this paper---by first truncating the kernel in a
  neighborhood of the singularity to obtain a uniformly bounded
  integrand, and then applying the dominated convergence theorem to the
  resulting translated truncated kernels, with the truncation error
  made arbitrarily small.
\end{remark}

\section{Appendix: Proof of Lemma~\ref{lem:und_int}\label{Ap1}}
We show that~\eqref{eq:PV_grad_inter} holds at $x = z$ for every fixed
$z\in \Omega$. To do this, for each such $z$ we first take a real
number $R_2>0$ such that the closure $K_{z,2} = \overline{B_{R_2}(z)}$
of the ball $B_{R_2}(z)$ centered at $z$ of radius $R_2$ is contained
in $\Omega$: $K_{z,2}\subset \Omega$. Further we take $R_1>0$ with
$R_1<R_2$ and we similarly define $K_{z,1} = \overline{B_{R_1}(z)}$; in
particular $K_{z,1}\subset K_{z,2}\subset \Omega$.  Further we select a
number $\delta>0$ such that $B_{\delta}(x) \subset K_{z,2} $ for all
$x \in K_{z,1}$, and for each $x \in K_{z,1}$ and $0 < \varepsilon < \delta$ we
define
\begin{equation}\label{eq:F_def_AP1}
  H_\varepsilon(x)=\int_{\Omega\setminus B_{\varepsilon}(x)} w(y)|x-y|^{-n-2s+2} dy.
\end{equation}

As shown in what follows, we have (i)~$H_\varepsilon$ is a
differentiable function of $x$ for $x \in K_{z,1}$, (ii)
$\frac{\partial H_\varepsilon}{\partial x_i}$ converges uniformly as
$\varepsilon \to 0$ for all $x \in K_{z,1}$ and for each $i$
($1\leq i\leq n$), and (iii) the relation
\begin{equation}\label{eq:AP1_main1}
  \lim_{\varepsilon \to 0} \frac{\partial H_\varepsilon}{\partial x_i} (x)=P.V \int_{\Omega} w(y) \frac{\partial }{\partial x_i} |x-y|^{-n-2s+2}dy
\end{equation}
holds for all $x\in K_{z,1}$ and for each $i$. Recognizing that
interchanging the limit and differentiation on the left-hand side
of~\eqref{eq:AP1_main1} yields the desired
expression~\eqref{eq:PV_grad_inter}, since
\[
\lim_{\varepsilon \to 0} H_\varepsilon(x)
= H_0(x)
= \int_{\Omega} w(y)\,|x-y|^{-n-2s+2}\,dy,
\]
and noting that, by~\cite[Thm.~7.17]{rudin1976}, the function $H_0$ is
indeed differentiable and the interchange of the limit and
differentiation operations is justified provided conditions~(i)
and~(ii) hold, we see that the proof of the lemma will be
complete once conditions~(i), (ii), and~(iii) are verified.

\begin{figure}[H]
\centering
\centerline{\includegraphics[width=13cm]{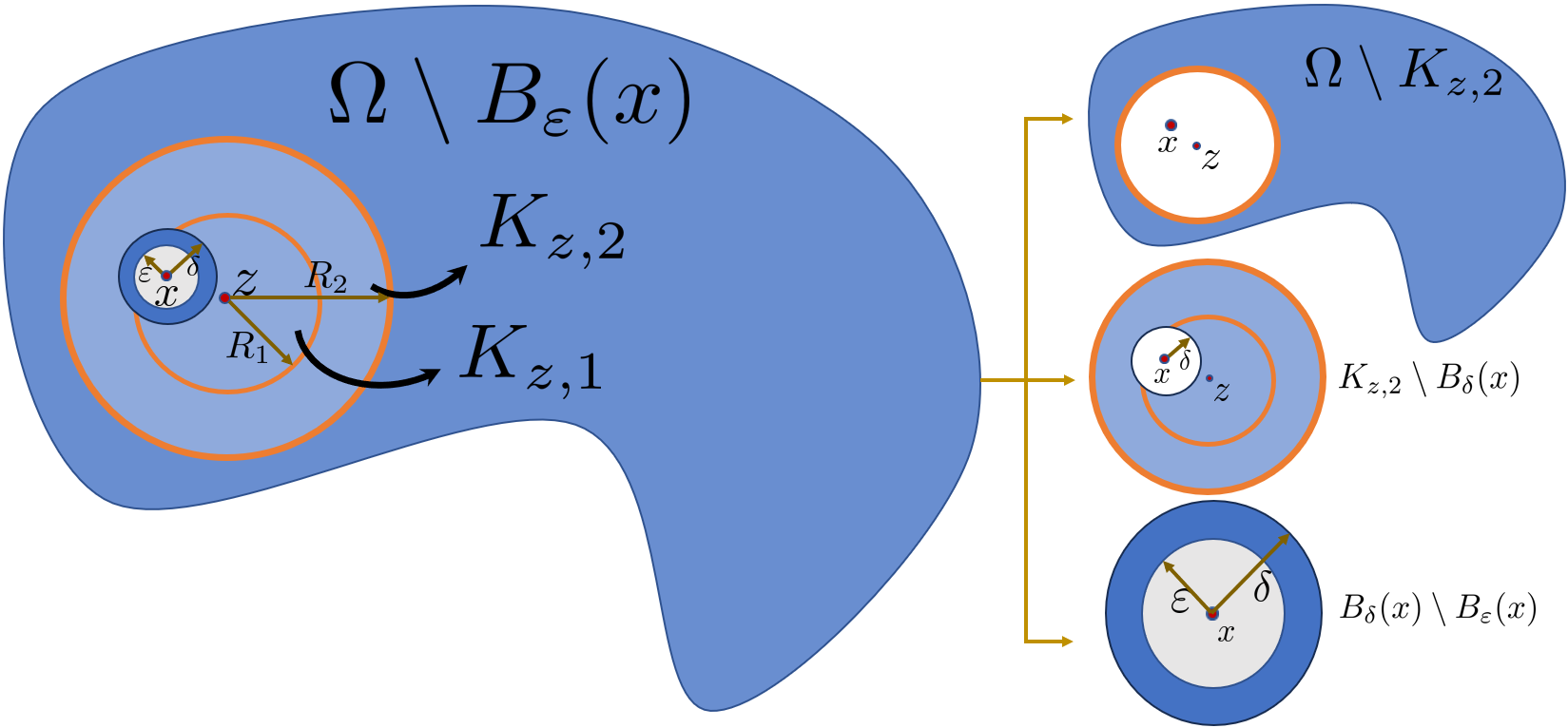}}
\caption{Illustration of $\Omega \setminus B_{\varepsilon}(x)$,
  $K_{z,1}$, and other auxiliary subsets of $\Omega$ employed in the
  proof of Lemma~\ref{lem:und_int}, with $x \in K_{z,1}$ and
  $0<\varepsilon<\delta$.}
\end{figure}
We now establish point~(i) and lay the groundwork for the proof of
points~(ii) and~(iii). To this end, we first decompose the integration
domain on the right-hand side of~\eqref{eq:F_def_AP1} as
$\Omega \setminus B_\varepsilon(x) = \left( \Omega \setminus K_{z,2}
\right) \cup \left( K_{z,2} \setminus B_\delta(x) \right) \cup \left(
  B_\delta(x) \setminus B_\varepsilon(x) \right)$, and we denote the
corresponding integrals by $\widetilde{H}(x)$, $H_\delta(x)$, and
$H_{\delta,\varepsilon}(x)$, respectively. This yields the
decomposition
\begin{equation}\label{eq:F_three_parts_AP1}
  H_\varepsilon(x) = \left(\int_{\Omega \setminus K_{z,2}} + \int_{K_{z,2} \setminus B_\delta(x)} + \int_{B_\delta(x) \setminus B_\varepsilon(x)}\right) \frac{w(y)}{|x - y|^{n + 2s - 2}}\, dy = \widetilde{H}(x) + H_\delta(x) + H_{\delta,\varepsilon}(x).
\end{equation}
Since the integrand of $\widetilde{H}$ is smooth for $x \in K_{z,1}$ and $y \in \Omega \setminus K_{z,2}$, and since the corresponding integration domain is independent of $x$, it follows that $\widetilde{H}(x)$ is differentiable for $x\in K_{z,1}$. Moreover, the partial derivative of $\widetilde{H}$ with respect to $x_i$ can be computed by differentiating under the integral sign:
\begin{equation}\label{eq:F1_der_AP1}
\frac{\partial \widetilde{H}}{\partial x_i}(x) =  \int_{\Omega \setminus K_{z,2}} w(y)\frac{\partial }{\partial x_i}|x-y|^{-n-2s+2}dy.
\end{equation}
The integral $H_{\delta,\varepsilon}$, in turn, can be expressed,
using spherical coordinates around $x$, as an integral over the unit
sphere $\partial B_1(0)$:
\begin{equation}\label{eq:third_int_polar_AP1}
H_{\delta,\varepsilon}(x) = \int_{\partial B_{r}(x)} \int_{\varepsilon}^{\delta} w(y) r^{-n - 2s + 2}  dr dS_{y} = \int_{\partial B_{1}(0)} \int_{\varepsilon}^{\delta} r^{1 - 2s} w(r\xi + x) dr dS_\xi.
\end{equation}
Since $w \in C^1(\Omega)$, the right-hand integral clearly displays
$H_{\delta,\varepsilon}$ as a differentiable function of $x$ for
$x\in K_{z,1}$, and it shows that its partial derivative with respect
to $x_i$ is given by
\begin{equation}\label{eq:F3_der_AP1}
\frac{\partial H_{\delta,\varepsilon}}{\partial x_i}(x) = \int_{\partial B_{1}(0)} \int_{\varepsilon}^{\delta} r^{1 - 2s} D_i w(r\xi + x) dr dS_\xi = \int_{ B_\delta(x) \setminus B_\varepsilon(x)} D_iw(y)|x-y|^{-n-2s+2} dy,
\end{equation}
where $D_iw$ denotes the partial derivative of $w$ with respect to the
$i$-th coordinate. The integral $H_\delta(x)$, finally, is handled by
using once again spherical coordinates centered at $x$.  To do this,
for a unit vector $\xi\in\partial B_{1}(0)$ we call $\rho_x(\xi)$ the
value of the radial parameter $r$ for which the ray
$\{r\xi + x\ |\ r>0 \}$ intersects the sphere
$\partial K_{z,2} = \{t\in \mathbb{R}^n\ | \ |t - z| = R_2 \}$, that
is to say,
\begin{equation}\label{eq:rho_def_AP1}
  \rho_x(\xi) = \mbox{positive solution}\quad r\quad\mbox{of the equation}\quad | x+r\xi-z|=R_2;
\end{equation}
note that the $z$-dependence of the quantity $\rho_x(\xi)$ is suppressed
in the notation since $z$ is held constant as $x$ and $\xi$ vary.
Clearly, the map $\gamma_x:\partial B_1(0)\to\partial K_{z,2}$ defined by
\begin{equation}
  \label{eq:param}
  \gamma_x(\xi) = x+\xi\rho_x(\xi)
\end{equation}
provides a smooth parametrization of the sphere $\partial
K_{z,2}$. For future use in this proof we note that, employing the
expression for the solution of a quadratic equation, the relation
\begin{equation}\label{eq:gamma_x_exp}
\rho_x(\xi)   =\alpha_x(\xi)-\beta_x(\xi),\quad\text{where}\quad    \alpha_x(\xi) = \sqrt{R_2^2 - |x-z|^2 + \beta_x^2(\xi)}\quad \text{and} \quad \beta_x(\xi) = (x-z)\cdot \xi
\end{equation}
results. Using~\eqref{eq:rho_def_AP1}, the integral $H_\delta$, may be
expressed in the spherical-coordinate form
\begin{equation}\label{eq:F_2_rexpressed_AP1}
    H_\delta(x) = \int_{\partial B_{1}(0)}\int_{\delta}^{\rho_x(\xi)}r^{1-2s} w(r\xi+x) dr dS_\xi.
\end{equation}
Since $\rho_x(\xi)$ is a $C^{\infty}$ function of $(x,\xi)$ for
$x \in K_{z,1}$ and $\xi\in\partial B_1(0)$,  Leibniz integral rule tells
us that $H_\delta$ is differentiable with respect to $x_i$ and
\begin{equation}\label{eq:F2_der_AP1}
    \frac{\partial H_\delta}{\partial x_i}(x) =  I^1(x) + \int_{K_{z,2} \setminus B_\delta(x)}D_iw(y)|x-y|^{-n-2s+2} dy,
\end{equation}
where the boundary term $I^1$ is given by 
\begin{equation}\label{eq:I1_def_AP1}
    I^1(x) = \int_{\partial B_{1}(0)} (\rho_x(\xi))^{1-2s} w(x+\rho_x(\xi) \xi)\frac{\partial \rho_x}{\partial x_i}(\xi) dS_\xi.
\end{equation}
Having shown that all three terms in~\eqref{eq:F_three_parts_AP1} are
differentiable with respect to $x_i$ for $x\in K_{z,1}$ it follows
that so is $H_\varepsilon$, and the proof of point~(i) is
complete.

To establish point~(ii), in turn, in view of
equation~\eqref{eq:F_three_parts_AP1}, and since the terms $\widetilde{H}$ and
$H_\delta$ in that equation are independent of $\varepsilon$, it
suffices to show that
$\partial H_{\delta,\varepsilon}(x)/\partial x_i$ converges uniformly
as $\varepsilon \to 0$. To this end we note that, in view
of~\eqref{eq:F3_der_AP1}, the limit
$\lim_{\varepsilon \to 0}\frac{\partial
  H_{\delta,\varepsilon}}{\partial x_i}(x)$ exists for $x\in K_{z,1}$
(since the corresponding integrand is integrable), and
\begin{equation}
\left| \lim_{\varepsilon \to 0}\frac{\partial H_{\delta,\varepsilon}}{\partial x_i}(x) - \frac{\partial H_{\delta,\varepsilon}}{\partial x_i}(x)\right| \le  \int_{\partial B_{1}(0)}\int_{0}^{\varepsilon} r^{1-2s} \left| D_iw(r\xi+x) \right| dr dS_\xi \le \underset{\substack{x\in K_{z,1} ,\, r \in (0,\varepsilon)\\ \xi \in \partial B_1(0)}}{\sup} |D_iw(r\xi+x)| \frac{\varepsilon^{2(1-s)}}{2(1-s)}\omega_n
\end{equation}
where $\omega_n$ denotes the surface area of the $n$-th dimensional
sphere. Therefore,
$ \frac{\partial H_{\delta,\varepsilon}}{\partial x_i}$ converges
uniformly as $\varepsilon \to 0$, and point (ii) follows.

To establish point~(iii), finally, we begin by adding
equations~\eqref{eq:F1_der_AP1}, \eqref{eq:F3_der_AP1} and
\eqref{eq:F2_der_AP1}, and thus obtain
\begin{equation}\label{eq:sum_ders}
\frac{\partial H_\varepsilon}{\partial x_i}(x) = I^1(x)+\int_{K_{z,2}\setminus B_{\varepsilon}(x)} D_iw(y)|x-y|^{-n-2s+2}dy + \int_{\Omega \setminus K_{z,2}} w(y)\frac{\partial }{\partial x_i}|x-y|^{-n-2s+2}dy.
\end{equation}
Integrating by parts the integral over
$K_{z,2}\setminus B_{\varepsilon}(x)$ in~\eqref{eq:sum_ders} we 
obtain
\begin{equation}\label{eq:int_by_parts_Kz_B}
    \int_{K_{z,2} \setminus B_{\varepsilon}(x)} D_iw(y)|x-y|^{-n-2s+2} dy = I^2(x) + I^3_{\varepsilon}(x)+\int_{K_{z,2} \setminus B_{\varepsilon}(x)} w(y)\frac{\partial }{\partial x_i}|x-y|^{-n-2s+2} dy
\end{equation}
where, using the $i$-th components $\frac{x_i-y_i}{\varepsilon}$ and
$\frac{y_i-z_i}{R_2}$ of the normals exterior to the domain
$K_{z,2} \setminus B_{\varepsilon}(x)$ along the surfaces
$\partial B_{\varepsilon}(x)$ and $\partial K_{z,2}$, respectively,
the boundary terms $I^2(x)$ and $I^3_{\varepsilon}(x)$ are given by
\begin{equation}\label{eq:I2_def_AP1}
I^2(x) = \int_{\partial K_{z,2}} w(y) |x-y|^{-n-2s+2}\frac{y_i-z_i}{R_2}dS_y
\end{equation}
and 
\begin{equation}\label{eq:I3eps}
I^3_{\varepsilon}(x) = \int_{\partial B_{\varepsilon}(x)}w(y) |x-y|^{-n-2s+2}\frac{x_i-y_i}{\varepsilon}dS_y = -\int_{\partial B_{\varepsilon}(0)} w(x+y) \frac{y_i}{\varepsilon^{n+2s-1}} dS_y,
\end{equation}
respectively. Using manipulations analogous to those
following~\eqref{eq:grad_kernel}, the right-hand side
of~\eqref{eq:I3eps} may be expressed in a form analogous
to~\eqref{eq:bd_int_z}, which, in particular, show that
$I^3_{\varepsilon}(x) \to 0$ as $\varepsilon \to 0$. Thus
\begin{equation}\label{eq:AP1_main1_prefinal}
\lim_{\varepsilon\to 0}\frac{\partial H_\varepsilon}{\partial x_i}(x) = I^1(x)+I^2(x) + P.V \int_{\Omega} w(y) \frac{\partial }{\partial x_i} |x-y|^{-n-2s+2}dy.
\end{equation}
and, thus, to complete the proof it suffices to show that $I^1+I^2 = 0$. 

To do this we first use~\eqref{eq:gamma_x_exp} and obtain the relations
$\frac{\partial \rho_x}{\partial x_i} = \frac{\partial
  \alpha_x}{\partial x_i} - \frac{\partial \beta_x}{\partial x_i}$,
$\frac{\partial \alpha_x}{\partial x_i}(\xi) = \frac{\beta_x(\xi) \xi_i -
  (x_i-z_i)}{\alpha_x(\xi)}$ and
$\frac{\partial \beta_x}{\partial x_i} = \xi_i$, and we thereby
re-express the integral $I^1$ in~\eqref{eq:I1_def_AP1} in the form
\begin{equation}\label{eq:I1_final_AP1}
    I^1(x) = -\int_{\partial B_{1}(0)} (\rho_x(\xi))^{1-2s} w(x+\rho_x(\xi) \xi)\frac{x_i-z_i+\rho_x(y) \xi_i}{\alpha_x(\xi)} dS_\xi.
\end{equation}
To express $I^2$ as an integral over $\partial B_{1}(0)$, on the other
hand, we employ the change of variables
\begin{equation}
  \widetilde{\gamma}_x : [0,\pi]^{n-2}\times [0,2\pi]\to \partial K_{z,2} \quad \mbox{defined by} \quad \widetilde{\gamma}_x(\theta) =  x + \rho_x(\xi(\theta))\, \xi(\theta)
 \end{equation}
 where
 $\xi :[0,\pi]^{n-2}\times [0,2\pi]\to \partial
 B_1(0)\subset\mathbb{R}^n$ denotes the spherical change of variables~\eqref{spher_1}. Using this change of variables the integral $I^2$ takes the form
\begin{equation}
  I^2(x) = \int_{[0,\pi]^{n-2} \times [0,2\pi]} (\rho_x(\xi(\phi,\theta)))^{-n-2s+2} w(x+\rho_x(\xi(\phi,\theta)) \xi(\phi,\theta))\frac{x_i-z_i+\rho_x \xi_i(\phi,\theta)}{R_2}  \sqrt{\text{det}(G_{\widetilde{\gamma}_x})} d\phi d\theta,
\end{equation}
where $G_{\widetilde{\gamma}_x}=G_{\widetilde{\gamma}_x}(\phi,\theta)$ and
$\sqrt{\det(G_{\widetilde{\gamma}_x})} \, d\phi d\theta$ denote the Gram
matrix and the surface element associated with the parametrization
$\widetilde{\gamma}_x$, respectively. As shown in
Lemma~\ref{lem:cov_nd_radial_func}, this area element can be expressed
in the form
\begin{equation}
  \sqrt{\text{det}(G_{\widetilde{\gamma}_x})} = \sqrt{\text{det}(G_S)}\left[(\rho_x(\xi))^{(n-1)} \sqrt{1 + |\nabla_\xi \ln(\rho_x(\xi))|^2-(\xi \cdot \nabla_\xi \ln(\rho_x(\xi)))^2}\right]_{\xi=\xi(\phi,\theta)}
\end{equation}
where $G_{S} = G_{S}(\phi,\theta)$ denotes the Gram determinant of the
spherical change of variables $\xi = \xi(\phi,\theta)$ and where
$ \sqrt{\text{det}(G_S)} = F(\phi)$ (eq.~\eqref{spher_elems}).  Since
$dS_\xi = \sqrt{\text{det}(G_S)}d\phi d\theta$ is the element of area
of $\partial B_1(0)$, it follows that
\begin{equation}
    I^2(x) = \int_{\partial B_{1}(0)} (\rho_x(\xi))^{1-2s} w(x+\rho_x(\xi) \xi)\frac{x_i-z_i+\rho_x \xi_i}{R_2} \sqrt{1+|\nabla_\xi \ln(\rho_x)|^2 - (\xi\cdot \nabla_\xi \ln(\rho_x))^2} dS_\xi.
\end{equation}
Further, since
$\nabla_\xi \rho_x = -(x-z)\frac{\rho_x(\xi)}{\alpha_x(\xi)}$,
$\xi \cdot \nabla_\xi \rho_x = -\frac{\beta_x(\xi) \rho_x(\xi)}{\alpha_x(\xi)}$
and $\alpha_x^2(\xi) = R_2^2 - |x-z|^2 + \beta_x^2(\xi)$,
\begin{equation}\label{eq:I2_final_AP1}
 I^2(x) = \int_{\partial B_{1}(0)} (\rho_x(\xi))^{1-2s} w(x+\rho_x(\xi) \xi)\frac{x_i-z_i+\rho_x \xi_i}{R_2} \sqrt{1+ \frac{|x-z|^2}{\alpha_x^2(\xi)}  - \frac{\beta_x^2(\xi)}{\alpha_x^2(\xi)}} dS_\xi = -I^1(x),
\end{equation}
which shows that $I^1(x)+I^2(x)=0$, as desired, and the proof is
complete.  \hfill\qedsymbol

\begin{lem}\label{lem:cov_nd_radial_func}
  Using the notations and conventions introduced above in this
  appendix, we have
 \begin{equation}
  \sqrt{\mathrm{det}(G_{\widetilde{\gamma}_x})} = \sqrt{\mathrm{det}(G_S)}\left[(\rho_x(\xi))^{(n-1)} \sqrt{1 + |\nabla_\xi \ln(\rho_x(\xi))|^2-(\xi \cdot \nabla_\xi \ln(\rho_x(\xi)))^2}\right]_{\xi=\xi(\phi,\theta)}.
\end{equation}
\end{lem}
\begin{proof} Using~\eqref{spher_1}--\eqref{spher_elems} together with
  the relations
  $\left(G_{\widetilde{\gamma}_x}\right)_{ij} = \partial_{i}
  \widetilde{\gamma}_x \cdot \partial_{j} \widetilde{\gamma}_x$,
  $\partial_{i}\rho_x = \nabla_\xi \rho_x(\xi) \cdot \partial_{i}\xi$,
  $\xi \cdot \xi = 1$ and $\xi \cdot \partial_{i} \xi =0$ for
  $1\leq i\leq n-1$ (where $\partial_{i} = \partial_{\phi_{i}}$ for
  $1 \le i \le n-1$ and $\partial_{i} = \partial_{\theta}$ for
  $i=n-1$ and , and where the slight abuse of notation
  $\partial_{i} \widetilde{\gamma}_x = 
      (\partial_{i}\rho_x) \xi+\rho_x \partial_{i}\xi$ has been introduced), we obtain
  \begin{equation}
     \left(G_{\widetilde{\gamma}_x}\right)_{ij}  = \partial_{i}\rho_x\partial_{j}\rho_x +\rho_x^2 \partial_{i}\xi \cdot \partial_{j}\xi.
   \end{equation}
   Since the Gram matrix $G_S\in \mathbb{R}^{(n-1)\times (n-1)}$ for the
   spherical change of coordinates is the diagonal matrix
   given by~\cite[p. 44]{waner2005}
   \[
     \left(G_S\right)_{ij} =  \partial_{i}\xi \cdot \partial_{j}\xi = \begin{cases}
       0 & \mbox{if } i\ne j,\\
       \prod_{k=1}^{j-1} \sin^2(\phi_k)& \mbox{if } i = j,
     \end{cases}
   \]
   letting
   $v = (\partial_{1}\rho_x,\dots,
   \partial_{n-1}\rho_x)^T$ we may then write
   $G_{\widetilde{\gamma}_x} = vv^T +\rho_x^2 G_{S}$. Using the matrix
   determinant relation~\cite[Lemma 1.1]{ding2007}
 \[ {\displaystyle \det(A + uv^T)=(1+v^T A^{-1} u )\,\det( A ),\,}
 \]
 which is valid for any invertible matrix $A$ and column vectors $u$
 and $v$, we obtain
 \begin{equation}\label{eq:det_G_gamma}
   \text{det}(G_{\widetilde{\gamma}_x}) = \rho_x^{2(n-1)} \cdot \text{det}(G_S) \cdot \left( 1 + \frac{1}{\rho_x^2} v^T G_S^{-1} v\right).
 \end{equation}
 Further, using
  \begin{equation}
\nabla_{\mathbb{S}^{n-1}} f =
\frac{\partial f}{\partial \phi_1} \, \widehat{\phi}_1 +
\frac{1}{\sin \phi_1} \frac{\partial f}{\partial \phi_2} \,  \widehat{\phi}_2 +
\frac{1}{\sin \phi_1 \sin \phi_2} \frac{\partial f}{\partial \phi_3} \,  \widehat{\phi}_3 +
\cdots +
\frac{1}{\sin \phi_1 \cdots \sin \phi_{n-2}} \frac{\partial f}{\partial \theta} \,  \widehat{\theta},
  \end{equation}
  where the normalized vectors
  $ \widehat{\phi}_j = \partial_{j}\xi/|\partial_{j}\xi|$
  ($1\leq j\leq n-2$) and
  $ \widehat{\theta} = \partial_{n-1}\xi/|\partial_{n-1}\xi|$ are
  orthonormal, the second scalar term within the parenthesis in
  equation~\eqref{eq:det_G_gamma} becomes
  \begin{equation}
       \frac{1}{\rho_x^2} v^T G_S^{-1} v = \frac{1}{\rho_x^2} \sum_{i=1}^{n-1}\left( \prod_{k=1}^{i-1} \frac{1}{\sin^2(\phi_k)}\right) \left( \partial_i \rho_x\right)^2  = \left| \frac{\nabla_{\mathbb{S}^{n-1}} \rho_x}{\rho_x}\right|^2 = \left| \nabla_{\mathbb{S}^{n-1}} \ln{\rho_x}\right|^2. 
  \end{equation}
  Since 
  \begin{equation}
      \left|\nabla_{\mathbb{S}^{n-1}} \ln \rho_x\right|^2=|\nabla \ln \rho_x-(\xi \cdot \nabla \ln \rho_x ) \xi|^2=|\nabla \ln \rho_x|^2-(\xi \cdot \nabla \ln \rho_x )^2,
  \end{equation}
  we obtain
  \begin{equation}
      \sqrt{\text{det}(G_{\widetilde{\gamma}_x})} =\sqrt{\text{det}(G_S)} \left[\rho_x^{(n-1)} \sqrt{1 + |\nabla \ln \rho_x|^2-(\xi \cdot \nabla \ln \rho_x)^2}\right],
    \end{equation}
    as desired, and the proof is complete.
 \end{proof}

\bibliographystyle{unsrt}

\end{document}